\documentclass[11p]{article}
\usepackage{pdflscape}
\usepackage[a4paper, total={6in, 9in}]{geometry}
\usepackage[english]{babel}
\usepackage[colorinlistoftodos]{todonotes}
\usepackage{amsmath,amsthm,amssymb}
\usepackage{amsfonts}
\usepackage{rotating}
\usepackage{tabularx}
\usepackage{braket}
\usepackage{cancel}
\usepackage[numbers]{natbib}
\usepackage{multirow}
\usepackage[normalem]{ulem}
\usepackage{graphicx}
\usepackage{cancel}
\usepackage{soul}
\usepackage{pgfplots}
\usepackage{xfrac}
\usepackage[linesnumbered, ruled,vlined]{algorithm2e}
\SetArgSty{textup}

\usepackage{hyperref}
\usepackage[nameinlink]{cleveref}
\newtheorem{thm}{Theorem}[section]

\newtheorem{lem}[thm]{Lemma}
\newtheorem{prop}[thm]{Proposition}
\theoremstyle{definition}
\newtheorem{defn}[thm]{Definition}
\theoremstyle{remark}
\newtheorem{rem}[thm]{Remark}

\newtheorem{assump}{ \rm \textbf{Assumption}}
\Crefname{assump}{assumption}{assumptions}
\numberwithin{equation}{section}
\Crefname{prop}{Proposition}{Propositions}
\crefname{prop}{proposition}{propositions}
\Crefname{thm}{Theorem}{Theorems}
\crefname{thm}{theorem}{theorems}
\Crefname{assump}{Assumption}{Assumptions}
\crefname{assump}{assumption}{assumptions}

\def \argmax{\mbox{argmax}}
\def \argmin{\mbox{argmin}}

\def \argmin{\mbox{argmin}}
\def \supp{\mbox{supp}}

\newcommand{\BlueUpdate}[1]{{#1}}

\newcommand{\qRCCD}{$q$-RCCD}
\newcommand{\AbAlg}{AB2CD}

\newcommand\Tstrut{\rule{0pt}{2.6ex}}    
\newcommand{\stkout}[1]{\ifmmode\text{\sout{\ensuremath{#1}}}\else\sout{#1}\fi}

\usepackage{tikz}
\usepackage{setspace}
\usepackage{comment}
\usepackage{autonum}
\pgfplotsset{compat=1.17}
\begin{document}

\title{On convergence of a $q$-random coordinate constrained algorithm for non-convex problems}
\author{A. Ghaffari-Hadigheh$^{1}$\thanks{Azarbaijan Shahid Madani University, Tabriz, Iran, {\tt Hadigheha@azaruniv.ac.ir}.
This research was initiated during the first author's sabbatical visit from Tilburg University.} \and L. Sinjorgo$^{2}$,\thanks{EOR Department, Tilburg University, The Netherlands, {\tt l.m.sinjorgo@tilburguniversity.edu}} 
\and R. Sotirov$^{3}$  \thanks{EOR Department, Tilburg University, The Netherlands, {\tt r.sotirov@uvt.nl}} }

\date{}

\maketitle
\begin{abstract}
We propose a random  coordinate descent algorithm for optimizing a non-convex objective function subject to one linear constraint and simple bounds on the variables. 
Although it is common use to update only two random coordinates simultaneously in each iteration of a coordinate descent algorithm, our
algorithm allows updating arbitrary number of coordinates. We provide a proof of convergence of the  algorithm. The convergence rate of the algorithm improves when we update more coordinates per iteration.
Numerical experiments on large scale instances of different optimization problems  show the benefit of updating many coordinates simultaneously.
\end{abstract}



\section{Introduction}
The prominence of so called ``Big Data" has given rise to new challenges for the field of optimization. Algorithms that aim to optimize large scale optimization problems should provide good solutions in reasonable time, be memory efficient and well scalable. 

Well suited methods for optimizing high-dimensional functions are coordinate descent (CD) methods. Coordinate descent algorithms solving  lasso penalized regression trace back to~\cite{fu1998penalized}.
In its most basic form, a CD method iteratively minimizes the objective function, by updating only a (strict) subset of variables per iteration. Variants of CD methods are distinguished by the selection procedure of the variables to be updated (e.g., cyclic, greedy, random), and consequently the method to update selected variables. Interested readers are referred to surveys~\cite{shi2016primer,wright2015coordinate} for more details on CD methods.
Coordinate descent methods find applications in various areas, including image denoising~\cite{mcgaffin2015edge}, sensor network localization~\cite{NishijimaNakata},  machine learning, such as Support Vector Machines~\cite{hsieh2008dual} and penalized regression~\cite{breheny2011coordinate}. 

Global convergence of random CD methods for convex objective functions was proved by \citeauthor{nesterov2012efficiency} \cite{nesterov2012efficiency}. \citeauthor{nesterov2012efficiency} \cite{nesterov2012efficiency} considered both constrained and unconstrained convex optimization problems, and provided an accelerated variant for the unconstrained case.
These results were  generalized in~\cite{richtarik2014iteration}, where \citeauthor{richtarik2014iteration} extended the analysis of \cite{nesterov2012efficiency} to include objective functions that are the sum of a smooth convex function and a simple nonsmooth block-separable convex function.
  In \cite{nesterov2014subgradient}, a class of huge-scale convex optimization problems with sparse subgradients has been considered.  The proposed method works well on problems with uniform sparsity of corresponding linear operators.

 Despite these results for convex optimization problems,  there are few developments on  solving large-scale non-convex problems.
  \citeauthor{patrascu2015efficient} \cite{patrascu2015efficient} introduced two 2-random coordinate descent algorithms for large scale structured non-convex optimization problems. 
The objective functions they consider are consisting of two terms where one is non-convex and smooth, and the other one is convex. One of the algorithms from~\cite{patrascu2015efficient} considers unconstrained problems, while the other one  singly linearly constrained problems.
Both algorithms are designed to update two block coordinates in each iteration. \citeauthor{cristofari2019almost}~\cite{cristofari2019almost} proposed an algorithm for optimization of a  non-convex function subject to one linear constraint and simple bounds on the variables.
The algorithm from~\cite{cristofari2019almost}   iteratively chooses a pair of cooridnates according to the almost cycling strategy, i.e.,  one variable is selected in a cyclic manner, and the other one not. 
This algorithm has deterministic convergence properties, which is not the case for random CD methods. 
Two random coordinate descent-based algorithms for solving a non-convex problem that allow updating two or more coordinates simultaneously, in each iteration, are introduced in~\cite{sotirov2020solving}. 
Extensive numerical tests in~\cite{sotirov2020solving} show better performance of the algorithms when more than two variables are simultaneously updated.  

Motivated by the algorithms from~\cite{sotirov2020solving},  we propose a random coordinate descent algorithm for minimization of a non-convex objective function subject to one linear constraint and bounds on the variables.  In each iteration, our $q$-random coordinate constrained descent (\qRCCD{}) algorithm randomly selects $q$ ($q \geq 2$) variables, with uniform probability.
The \qRCCD{} algorithm updates those variables based on the optimization of a convex approximation of the objective. In particular, updates of variables are based on a projected gradient method. For this purpose, we assume $q$-coordinate Lipschitz continuity of the gradient of the objective function. This assumption differs from the assumption on the  gradient of the  objective  from~\cite{sotirov2020solving}.
Moreover, we provide a proof of convergence of the \qRCCD{} algorithm, using techniques similar to those employed in \cite{patrascu2015efficient}.  
The rate of convergence for the expected values of an appropriate stationarity measure of the \qRCCD{} algorithm, coincides with the convergence rate of the algorithm from~\cite{patrascu2015efficient} when $q=2$ and blocks are of size one.
However, the performance of  the \qRCCD{} algorithm improves for $q$ larger than two. This  implies that  we improve on the work of \citeauthor{patrascu2015efficient} \cite{patrascu2015efficient}  when block sizes are of size one. We test the \qRCCD{} algorithm  for solving large scale instances of the densest $k$-subgraph  (D$k$S) problem  and the eigenvalue complementarity  (EiC) problem. 
Our numerical results show the benefit of updating more than two coordinates simultaneously. \BlueUpdate{Additionally, we compare the \qRCCD{} algorithm with an Alternating Direction Method of Multipliers (ADMM), see e.g.,~\cite{boyd2011distributed}, the projected gradient method (PGM)~\cite{calamai1987projected}, and the deterministic 2-coordinate descent algorithm by \citeauthor{beck20142} \cite{beck20142}. Our results clearly show that the \qRCCD{} algorithm is superior for most of the considered instances.}

The rest of this paper is organized as follows. We formally introduce the \qRCCD{} algorithm and the considered class of optimization problems in \Cref{Section_Qrcca}. We define a {stationarity measure} in \Cref{Section_Opt_Measure}, which we use to prove the convergence of the \qRCCD{} algorithm in \Cref{Section_ConvergenceAnalysis}. Numerical results are presented in Section \ref{sect:numerics}.

 \section{A \texorpdfstring{$q$}{q}-random coordinate constrained descent algorithm}
\label{Section_Qrcca}
Let  $f:\mathbb{R}^n\to \mathbb{R}$ be a non-convex differentiable  function,
and $a=(a_i) \in \mathbb{R}^n$ and $b \in \mathbb{R}$ given vectors.
Further, let $e=(e_i)\in \mathbb{R}^n$ and $g=(g_i)\in \mathbb{R}^n$ be vectors such that $e\leq g$.
We consider the following   non-convex optimization problem with one linear constraint and bounded variables: 
\begin{equation}
\label{general}
 \min \left \{ f(x) ~:~\displaystyle a^\top x=b, ~e_i\leq x_i\leq g_i, ~i\in [n] \right \}.
\end{equation}
Here $[n]$ denotes the set $\{1,\ldots,n\}$. If $a_i \neq 0$ for all $i \in [n]$, then without loss of generality, one may assume that $a_i = 1$ for all $i \in [n]$ by a simple rescaling argument, see e.g., \cite[Section 2.1]{beck20142}.

In many  applications, function $f$ is of the form $f(x)=x^\top Q x + c^\top x$ where $Q=(q_{ij})$ is an indefinite matrix and $c=(c_i)$ a given vector.
Examples are the continuous densest subgraph problem (see  Section \ref{Section_K_densest_subgraph}), the continuous quadratic knapsack problem, support vector machine training, the Chebyshev center problem, etc.

We denote the feasible  set of the optimization problem  \eqref{general}  by $P$. That is \begin{align}
\label{feasibleP}
P = \left \{ x\in \mathbb{R}^n: ~a^\top x=b, ~e_i\leq x_i\leq g_i, ~i\in [n] \right \}.
\end{align}
The first order necessary optimality condition for local solutions of \eqref{general} are given in the sequel.\\ 

\noindent\textbf{Necessary Optimality Condition}.
  If $x^* \in P$ is a local { minimum} of \eqref{general}, then there exist a $\lambda^* \in \mathbb{R}$ ,
and vectors  $\gamma^*, \delta^*\in \mathbb{R}^n_+$  with
\begin{eqnarray}
  \nabla f(x^*) +\lambda^*  a -\gamma^* +\delta^*&=&0, \label{Necessary_Optimality_Condition} \\
  a^\top x^* &=& b, \\
  g_j - x^*_j &\geq& 0, ~ j \in [n]\\
  x^*_j - e_j &\geq& 0,~ j \in [n]\\
   \gamma^*_j (x^*_j-e_j)&=&0,~ j \in [n] \label{Necessary_Optimality_Condition-1}\\
 \delta^*_j(g_j-x^*_j)&=&0, ~ j \in [n].\label{Necessary_Optimality_Condition-2}
\end{eqnarray}
{We refer to a vector $x^*$ satisfying \eqref{Necessary_Optimality_Condition}--\eqref{Necessary_Optimality_Condition-2} as a {\em stationary point}.}
We aim to devise an iterative algorithm with a  feasible descending random direction that converges to a {stationary point} of \eqref{general}.
Before we present our algorithm, we state the following assumption that applies to the whole paper.
\begin{assump}\label{assump2} 
Let  $q\in \mathbb{N}$, $q\geq 2$ be given.
  The function $f$  has $q$-coordinate Lipschitz continuous gradient at $x\in \mathbb{R}^n$,
i.e., for any $J  \subseteq [n]$ with $|J|=q$,  there exists a constant $L_{J} > 0$ such that:
\begin{equation}\label{q-coordinate}
  \| \nabla_J f(x+ s_J)- \nabla_J f(x) \| \leq L_J \|s_J\|,
\end{equation}
where $\| \cdot \|$ denotes the Euclidean norm,
$s_J \in \mathbb{R}^n$ is a vector with  elements  $s_{j}$ for $j \in J$, and zeros for $j\not \in J$. Analogously,
  $\nabla_J f(x)\in \mathbb{R}^n$ denotes the vector consisting of
 $\frac{\partial f}{\partial x_{j}}$ for $j\in J$ and zeros for $j\not \in J$. \label{gradianJ}
 
 To simplify notation,  we  also use  $\nabla_J f(x)$  for the  projection of the gradient onto  the subspace  $\mathbb{R}^q$ identified by $J$,
  when it is clear from the context what we mean.
\end{assump}
 By \Cref{assump2} and from ~\cite[Lemma 1.2.3]{nesterov2018lectures},  we have
\[ 
  |f(x+s_J)-f(x) - \langle \nabla_J f(x), s_J  \rangle| \leq \frac{L_J}{2} \|s_J\|^2,~~\forall x\in \mathbb{R}^n, s_J \in \mathbb{R}^n,
\]
from where it follows that
\begin{equation}\label{identical-final-result-of-q-coordinate}
  {f(x+s_J) \leq f(x) + \langle \nabla_J f(x), s_J  \rangle  +  \frac{L_J}{2} \|s_J\|^2,
  ~~\forall x\in \mathbb{R}^n, s_J \in \mathbb{R}^n.}
  \end{equation}
We use the quadratic approximation of $f$ in \eqref{identical-final-result-of-q-coordinate}  to obtain {descending} directions of our algorithm. In particular, in each iteration of our $q$-random  coordinate constrained descent algorithm  we update   $q\geq 2$ random coordinates by exploiting the right hand side of  \eqref{identical-final-result-of-q-coordinate}.

Let $x^{m}$ be a feasible solution of problem \eqref{general} in iteration $m$ of the \qRCCD{} algorithm.  Let $J_m=\{j^m_{1}, \ldots, j^m_{q}\}
 \subseteq [n]$ with $|J_m| = q$ be a set of random coordinates that needs to be updated simultaneously in step $m$. 
Our $q$-random coordinate constrained update in $m$-th iteration is  as follows: \begin{equation}\label{update-step}
   d_{J_m}(x^m) := \left\{ \begin{array}{lcc}
                 u^m_j(x^m) - x^m_j  & \mbox{ if } &  j \in J_m, \\[2mm]
                  0 &\mbox{ if }  & j \not\in J_m,
               \end{array}\right. 
 \end{equation}
 where $u^m(x^m)$ is the optimal solution of a convex optimization problem.  To simplify notation, we sometimes write $u^m$ instead of $u^m(x^m)$, and $d_{J_m}$ instead of $ d_{J_m}(x^m) $.
In view of \eqref{identical-final-result-of-q-coordinate},  we define {$u^m(x^m)$} as follows:
 \begin{equation}\label{update_subproblem}
\hspace{-1mm} \begin{array}{lrl}
     u^m(x^m) :=  & \displaystyle \underset{\bar{u} \in \mathbb{R}^q}{\argmin} &\displaystyle \sum_{j\in J_m} \frac{\partial f(x^m )}{\partial x_j} (\bar{u}_j -x^m_j)
    { + }\frac{L_{J_m}}{2}\sum_{j\in J_m}(\bar{u}_j -x^m_j)^2       \\[6mm]
    & \mbox{s.t.} & \displaystyle\sum_{j\in J_m}a_j \bar{u}_j =b- \sum_{j\notin J_m}a_j x^m_j = \sum_{j \in J_m}a_j x^m_j \\[6mm]
    &  &  e_j\leq \bar{u}_j \leq g_j, \qquad \qquad \forall j \in J_m.
 \end{array}
  \end{equation}
Let us consider \eqref{update_subproblem} in more detail. The objective function can be rewritten as
  \begin{align}
      \underset{\bar{u} \in \mathbb{R}^q}{\argmin} \quad  2 \Big\langle \frac{1}{L_{J_m}} \nabla_{J_m} f(x^m) -x_{J_m}^m, \bar{u} \Big\rangle + \bar{u}^\top \bar{u}= \underset{\bar{u} \in \mathbb{R}^q}{\argmin} ~\Big\| \bar{u} - \left (x_{J_m}^m - \frac{1}{L_{J_m}} \nabla_{J_m} f(x^m) \right ) \Big\|^2.
  \end{align}
  This shows that 
  \begin{align}
        \label{eqn_closedFormUpdate}
        u^m(x^m) = \Big( x_{J_m}^m - \frac{1}{L_{J_m}} \nabla_{J_m} f(x^m) \Big)_{\mathcal{P}},
  \end{align}
  where  $(\cdot)_\mathcal{P}$ denotes the orthogonal projection onto the feasible set of \eqref{update_subproblem}. As this set is convex, computing such a projection is equivalent to solving a convex quadratic programme, which can be done in polynomial time (see~\cite{KozTarKha79}). However, some specific vectors $a$, $e$ and $g$,  allow for the use of more efficient projection algorithms, as we will see in \Cref{sect:numerics}. \\

We update $x^m$ as follows:
\begin{equation}
\label{updateiterate}
x^{m+1} = x^m+ d_{J_m}(x^m), 
\end{equation}
and obtain the next feasible point of the optimization problem  \eqref{general}

\begin{rem}
The objective in \eqref{update_subproblem} differs from the objective in~\cite{sotirov2020solving}, where the following objective is used:
$$\sum_{j\in J_m} \frac{\partial f(x^m )}{\partial x_j} (\bar{u}_j^m -x^m_j)     {+} \sum_{j\in J_m} \frac{L_j}{2} (\bar{u}_j^m -x^m_j)^2,  $$     
as well as an assumption that $f$ is coordinatewise Lipschitz continuous with constants $L_j$ for $j\in J_m$.
\end{rem}

Let us summarize  the \qRCCD{} algorithm.
For fixed $q$  such that  $2 \leq  q < n $, the $q$-random coordinate  constrained descent algorithm is as follows.
The algorithm starts with a feasible solution for \eqref{general}. Then, chooses $q$ coordinates randomly by uniform distribution on $[n]$.
For this subset, a $q$-coordinate Lipschitz constant is calculated that satisfies \eqref{q-coordinate} and the auxiliary  quadratic
convex optimization problem~\eqref{update_subproblem} is solved.
If no stopping criteria  is satisfied, the algorithm calculates a new  feasible solution of \eqref{general} by using \eqref{updateiterate}, and continues.
The algorithm may stop if the difference in two consecutive objective values is less than a pre-specified tolerance.
For more details on the algorithm see the pseudo-code given by \Cref{Alg1}.

{In light of \eqref{eqn_closedFormUpdate}, the \qRCCD{} algorithm can be considered as a combination of coordinate descent and PGM \cite{calamai1987projected}. In fact, in the extreme case $q = n$, the \qRCCD{} algorithm is deterministic and equivalent to PGM with constant stepsize $1/L_{J}$ (note that $L_{J}$ is constant when $q = n$).}

\begin{algorithm}[ht!]
\caption{The \qRCCD{} algorithm for the non-convex  problem \eqref{general}} \label{Alg1}
\textbf{Input:} A feasible initial  solution $ x^0$ of \eqref{general}, {and parameter $q \in \mathbb{N}$, $q \geq 2$.} \\
\textbf{Initialize}: $m \leftarrow 0$. \\
 \While{stopping criteria not satisfied}{
Select  $J_m \subseteq [n]$, $|J_m|=q$ coordinates randomly by uniform distribution on $[n]$. \\
Calculate appropriate $L_{J_m}$. \\
{Compute}  {\eqref{eqn_closedFormUpdate}} and use \eqref{update-step}  to obtain $d_{J_m}(x^m) \in \mathbb{R}^n$. \\
Update $x^{m+1}$ by using \eqref{updateiterate}. \\
$m \leftarrow  m+1$. \\
}
\textbf{Output:} $f(x^{m-1})$
\end{algorithm}

Observe that the vector $u^m\in \mathbb{R}^q$ is an optimal solution of the convex optimization
 problem in~\eqref{update_subproblem} if and only if there exist $\lambda \in \mathbb{R}$,
and $\gamma, \delta \in \mathbb{R}_+^q$ such that 
\begin{eqnarray}
   \frac{\partial f(x^m)}{\partial x_j} {+ L_{J_m} (u^m_j-x^m_j)} +\lambda a_j {- \gamma_j + \delta_j }=0,&&   j\in J_m\label{update_subproblem-first_optimality condition} \\
        \sum_{j\in J_m}a_ju^m_j=~~~&&\hspace{-9mm} {b- \sum_{j\notin J_m}a_jx^m_j}  \label{update_subproblem-fourth_optimality condition}\\
    e_j \leq u^m_j \leq g_j,&&  j\in J_m  \label{update_subproblem-last_optimality condition}\\
    \gamma_j ( u^m_j - e_j)=0,& & j\in J_m \label{update_subproblem-second_optimality condition}\\
    \delta_j (g_j-u^m_j) =0,&& j\in J_m. \label{update_subproblem-third_optimality condition}
 \end{eqnarray}
 
Note that for $x^m \in P$, see \eqref{feasibleP}, a feasible solution $u^m$ for  \eqref{update_subproblem} leads to 
\begin{eqnarray}\label{eq3-1-1}
\begin{array}{rcl}
  \langle a_{J_m},u^m-x^m_{J}\rangle&=&\displaystyle\sum_{j\in J_m}a_ju^m_j - \sum_{j\in J_m}a_jx^m_j  \\[4mm]
  &=&b- \displaystyle \sum_{j\notin J_m}a_jx^m_j - \sum_{j\in J_m}a_jx^m_j= b- \sum_{j=1}^n a_j x^m_j =0,
\end{array}
  \end{eqnarray} 
where $a_{J_m}$ (resp.,~$x^m_{J}$) is the sub-vector of $a$  (resp.,~$x^m$) with indices in $J_m$.  Note that we omit subscript $m$ on $J_m$ in 
 $x^m_{J_m}$, and write $x^m_{J}$.
By abuse of notation, we  write \eqref{eq3-1-1} as
$$
\langle a_{J_m},d_{J_m}\rangle = 0.
$$
 
\subsection{Notation}
We define the set of all index sets of size $2 \leq q < n$ as
\begin{align}
    \label{eqn_Jdef}
    \mathcal{J}_q := \left \{ J~:~ J\subseteq [n],~  |J|=q \right \},
\end{align}
and often omit the subscript $q$ for brevity.
We assume that any two index sets $J, \, J^\prime \in \mathcal{J}_q$ have identical probability to be selected. We set
\begin{align}
\label{eqn_numberZ}
    z := | \mathcal{J}_q | = \binom{n}{q}.
\end{align}
By our previous assumption, each index set $J \in \mathcal{J}_q$ can be chosen with  probability $1/{z}$.
Note that for $|J|=q$ and fixed $i$ we have $\mu = \binom{n-1}{q-1}$ numbers of different $J$'s such that each of them contains element $i$.
 Therefore, given Lipschitz constants $L_J > 0$, $J \in \mathcal{J}_q$, we define the vector $\Gamma =(\Gamma_i) \in \mathbb{R}^n$ where:
  \begin{equation}\label{definition-of-gamma_i}
    \Gamma_i := \frac{1}{\mu}\sum_{\{ J:~ i\in J \} }  L_J,~~~i\in [n].
  \end{equation}
  We further define  the diagonal matrix
  \begin{align} \label{defintionofD}  
  D_{\Gamma} := {\rm Diag}(\Gamma) \in \mathbb{R}^{n\times n}. 
  \end{align}
  Here ${\rm Diag}$ is the operator
 that maps a vector into the diagonal matrix whose diagonal elements correspond to the elements of the input vector.
 Further, we define  the following pair of primal-dual norms:
\begin{equation}\label{definition-of-norms}
 \begin{array}{rcl}
   \|x\|_{\Gamma} & = &(x^\top D_{\Gamma}x)^{\frac{1}{2}} =\| D_\Gamma ^{\frac{1}{2}} x \|,
   ~~~~ \forall x\in \mathbb{R}^n, \\[4mm]
   \|y\|^*_{\Gamma} & = &    (y^\top D^{-1}_{\Gamma}y)^{\frac{1}{2}},
   ~~~ \forall y\in \mathbb{R}^n.
 \end{array}
 \end{equation}
We denote by $\textbf{1}$ the vector of all-ones. For comparison purpose of our method  with the one of \cite{patrascu2015efficient},  we follow a similar notation and methodology.

\section{Stationarity measure}
\label{Section_Opt_Measure}

In this section we introduce a function whose optimal solution  is used to identify a stationary point  of  optimization problem \eqref{general}.

Consider the polyhedron
\begin{align}\label{setS}
S(x) = \left \{s \in \mathbb{R}^n ~: ~a^\top s = 0, ~e \leq s+x \leq g \right \},
\end{align}
 for some $x \in \mathbb{R}^n$, and  the corresponding local polyhedron with respect to $J \subseteq [n]$, i.e., 
\begin{align}\label{setSJ}
S_J(x) = \{s \in S(x) ~: s_j = 0,~ \forall j\notin J \}.
\end{align}
Note that for any vector $s\in S(x)$ we have $a^\top (x +s) = b$ when $x\in P$, see \eqref{feasibleP}.  
Thus, any vector $s\in S_J(x)$ where $|J|=q$ plays the role of a feasible update direction in each iterate of \Cref{Alg1}.

For any  $x \in P$,   $\alpha>0$, and vector $\Gamma \in \mathbb{R}^n_+$,
we define  the convex function $\psi_{\alpha \Gamma} (\cdot~; x): \mathbb{R}^n   \rightarrow  \mathbb{R}$, as follows:
\begin{equation}\label{definition-of-psi}
\psi_{\alpha \Gamma} (s; x) := f (x) + \langle \nabla f (x), s\rangle + \frac{\alpha}{2} \|s\|^2_{\Gamma}.
   \end{equation}
Note that  when $x$ is a stationary point of \eqref{general} and $s=0$, all terms of \eqref{definition-of-psi} vanish but the first one.
Thus, $\psi_{\alpha \Gamma} (s; x)$ plays the role of an overestimator of the function $f(\cdot)$
in the vicinity of a {stationary point}, see \eqref{identical-final-result-of-q-coordinate}.

Let us consider the following minimization problem:
\begin{equation} \label{auxmax}
\min_{s\in S(x)}~ \psi_{ \alpha  \Gamma} (s; x).
\end{equation}
Since $\psi_{\alpha \Gamma}(\cdot~;x)$ is a convex function,  $s\in S(x)$   is a global optimum of the minimization problem \eqref{auxmax}  if and only if there exist a $\lambda \in \mathbb{R}$, $\gamma, \delta\in \mathbb{R}^n_+$ such that
\begin{equation}\label{optimality-conditions-of-maximizer-for-psi}
 \nabla f(x) { + \alpha D_{\Gamma} } \cdot s + \lambda a  { -\gamma +\delta} =0,
\end{equation}
\begin{align}\label{psi_NecessAndSuff}
\gamma^\top(s+x-e)=0,~\delta^\top(g-s-x)=0, ~a^\top s=0, ~e \leq s+x \leq g.  
\end{align}
Thus,
\begin{equation}\label{definition-of-maximizer-for-psi}
  d_{ \alpha \Gamma} (x) := ~\displaystyle\mathrel{\mathop{\argmin}_{s\in S(x)}}\psi_{ \alpha  \Gamma} (s; x)
  = \frac{1}{\alpha} D^{-1}_{\Gamma} \left ({ -\nabla f(x) -\lambda a +\gamma - \delta}\right),
\end{equation}
provided that $d_{ \alpha \Gamma} (x) \in S(x)$ and satisfies \eqref{psi_NecessAndSuff}.
\medskip

Let $x \in P$ be a feasible solution of the non-convex optimization problem  \eqref{general}.
We introduce the following \textit{stationarity measure}, see also  \cite{patrascu2015efficient},
\begin{equation}\label{optimality_measure}
  M_q(x, \Gamma) = \|d_{\mu \Gamma}(x )\|_{\Gamma},
\end{equation}
where $\mu = \binom{n-1}{q-1}$, $\Gamma$  is defined in \eqref{definition-of-gamma_i},
and $d_{\mu \Gamma}(x )$ is given in
\eqref{definition-of-maximizer-for-psi}.

\begin{lem}
\label{lem3.1}
Let $2\leq q< n$, and  ${\Gamma}\in \mathbb{R}^n_+$ be  defined as in \eqref{definition-of-gamma_i}.
 Then,    $x^* \in P$
  is a stationary point  of non-convex optimization problem  \eqref{general} if and only if  $M_q(x^*, {\Gamma}) = 0$.
\end{lem}
\begin{proof}
Suppose that  $M_q(x^*, \Gamma)=0$. Then, it follows from \eqref{psi_NecessAndSuff} that $x^*$ is a stationary point for optimization problem  \eqref{general}, see also \eqref{Necessary_Optimality_Condition}--\eqref{Necessary_Optimality_Condition-2}.

Conversely,  suppose that $x^*$ is a stationary point for problem \eqref{general}. Since $x^*$ is a stationary point, it follows from \eqref{Necessary_Optimality_Condition} that there exist
$\lambda^* \in \mathbb{R}$,  $\gamma^*, \delta^*\in \mathbb{R}^n_+$ such that
\[
 \nabla f(x^*) +\lambda^*  a -\gamma^* +\delta^* = 0.
\]
The solution  of \eqref{auxmax} for $x=x^*$ and $\alpha:= \mu= \binom{n-1}{q-1}$, is denoted by $d_{\mu \Gamma} (x^*) \in S(x^*)$.
From \eqref{optimality-conditions-of-maximizer-for-psi} it follows that there exists a $\overline{\lambda} \in \mathbb{R}$ and $\overline{\gamma}, \, \overline{\delta} \in \mathbb{R}^n_+$ such that
\[
\nabla f(x^*) +  \mu D_{ \Gamma} \cdot  d_{ \mu \Gamma} (x^*)  +\overline{ \lambda}  a {-\overline{\gamma} +\overline{\delta} } =0.
\]
By subtracting the previous two equations, we obtain
\[
 \mu D_{ \Gamma} \cdot  d_{ \mu \Gamma} (x^*) + {(\bar{ \lambda} - \lambda^*)}a + (\gamma^* - \overline{\gamma}) - ( \delta^* - \overline{\delta})= 0.
\]
Taking the inner product with $d_{\mu \Gamma} (x^*)$ and using that $a^\top d_{ \mu \Gamma} (x^*)=0$, we obtain
\begin{align}
    \label{eq_part1}
    \mu \| d_{ \mu \Gamma} (x^*) \|^2_{\Gamma} + d_{ \mu \Gamma} (x^*) ^\top(\gamma^* - \delta^* - \overline{\gamma} +\overline{\delta}) = 0.
\end{align}
From the optimality conditions \eqref{psi_NecessAndSuff} we have
\begin{align}
    \overline{\gamma}^\top ( d_{ \mu \Gamma} (x^*) + x^* - e) = \overline{\delta}^\top (d_{ \mu \Gamma} (x^*) + x^*-g)=0,
\end{align}
and thus
\begin{align}
    \label{eq_StationarityPsi}
    \overline{\gamma}^\top  d_{ \mu \Gamma} (x^*) = \overline{\gamma}^\top (e-x^*) \quad \mbox{   and   } \quad \overline{\delta}^\top d_{ \mu \Gamma} (x^*) = \overline{\delta}^\top(g - x^*). 
\end{align}
Substituting \eqref{eq_StationarityPsi} in \eqref{eq_part1} and using that $e \leq d_{ \mu \Gamma} (x^*) + x^* \leq g$ yields
\begin{align}
    \mu \| d_{ \mu \Gamma} (x^*) \|_{\Gamma}^2 &= \overline{\gamma}^\top d_{ \mu \Gamma} (x^*) - \overline{\delta}^\top d_{ \mu \Gamma} (x^*) - \gamma^{*\top} d_{ \mu \Gamma} (x^*) + \delta^{*\top} d_{ \mu \Gamma} (x^*) \\
    &= \overline{\gamma}^\top (e - x^*) - \overline{\delta}^\top (x^* - g) - \gamma^{*\top} d_{ \mu \Gamma} (x^*) +\delta^{*\top} d_{ \mu \Gamma} (x^*) \\
    & \leq {0 + 0 } -  \gamma^{*\top}(e - x^*)  + \delta^{*\top} (g - x^*) \\
    &= 0,
\end{align}
from where it follows that $\| d_{ \mu \Gamma} (x^*) \|_{\Gamma} = M_q(x^*, \Gamma) =0$. 
\end{proof}
  
 We prove another lemma related to $d_{\alpha \Gamma}(x)$ and function $\psi$.
 \begin{lem}
     \label{lem_inProdIneq}
     Let $x \in P$, $\psi_{\alpha \Gamma} (\cdot~; x)$ as in \eqref{definition-of-psi}, and $d_{\alpha \Gamma}(x)$ according to \eqref{definition-of-maximizer-for-psi}. Then
    \begin{align}
        {\nabla \psi_{\alpha \Gamma} (d_{\alpha \Gamma}(x); x)^\top d_{\alpha \Gamma}(x) \leq 0.}
    \end{align}
 \end{lem}
 \begin{proof}
     From \eqref{optimality-conditions-of-maximizer-for-psi}, with $s =  d_{\alpha \Gamma}(x)$, and
     $\nabla \psi_{\alpha \Gamma}( s; x) =\nabla f(x) + \alpha D_{\Gamma} \cdot s$
     we have
     \begin{align}
          \nabla \psi_{\alpha \Gamma} (d_{\alpha \Gamma}(x); x) = - \lambda a + \gamma - \delta.
     \end{align}
     Taking the inner product with $d_{\alpha \Gamma}(x)$, and using that $a^\top d_{\alpha \Gamma}(x) = 0$, yields
     \begin{align}
        \label{eqn_inProdLemmaEq1}
         \nabla \psi_{\alpha \Gamma} (d_{\alpha \Gamma}(x); x)^\top d_{\alpha \Gamma}(x) = {\gamma^\top  d_{\alpha \Gamma}(x) - \delta^\top  d_{\alpha \Gamma}(x)}.
     \end{align}
    Nonpositivity of this expression then follows from substituting the first two equations from \eqref{psi_NecessAndSuff} in \eqref{eqn_inProdLemmaEq1}, where $s = d_{\alpha \Gamma}(x)$.
 \end{proof}

\section{Convergence analysis of the \texorpdfstring{\qRCCD{}}{q-RCCD}  algorithm}
\label{Section_ConvergenceAnalysis}
In this section we provide a convergence analysis of the \qRCCD{}   algorithm and its convergence rate.
We first define several terms and present known results that are needed later in the section.
\begin{defn}\label{stronglyConvex}
If the function  {$h$} is twice continuously differentiable, then it is strongly convex with parameter $\rho$ if and only if  {$\nabla ^{2}h(x)\succeq \rho I$}  for all $x$ in the domain.
An equivalent condition is 
\begin{align}
    { h(x_1) \geq h(x_2)+\nabla h(x_2)^\top(x_1-x_2)+{\frac {\rho}{2}}\|x_1-x_2\|^{2},}  
\end{align}
for every $x_1,x_2$ in the domain of  {$h$}.
\end{defn}

\medskip
Let $K \subseteq \mathbb{R}^n$ be an arbitrary subspace.
The set of indices corresponding to the nonzero coordinates in the vector $x\in K$ is called the \textit{support} of $x$, and denoted by $\supp(x)$.
For $s,s' \in \mathbb{R}^n $, we say that $s'$ is {\em conformal} to $s$ if 
 \begin{align}
    \label{eqn_conformalityConditions}
     \supp(s') \subseteq \supp(s), \quad s_j s'_j \geq 0,~ \forall j.
 \end{align}
The second condition in \eqref{eqn_conformalityConditions} states that  the nonzero components of $s'$ and the corresponding components of $s$  have the same signs.
A nonzero vector $s$ is an {\em elementary} vector of $K$  if $s\in K$ and there is no nonzero $s'\in K$ that is  conformal to $s$ and $\supp(s') \neq \supp(s)$. Any two elementary  vectors $s$ and $s'$  of $K$ with identical support must be (nonzero) scalar multiples of each other. 
Let us state the following well known result.

\begin{prop}[\cite{RockafellarTheEV}]\label{lem6}
Let $x$ be a nonzero vector in a subspace $K$ of $\mathbb{R}^n$  (or even generally in $\mathbb{R}^n$). Then, there exist elementary vectors $x_1, \ldots,x_r$  of $K$ such that $x=x_1+\cdots+ x_r$. They can be chosen such that each is  conformal to  $x$ and $r$ does not exceed the  dimension of $K$ or the number of elements in $\supp(x)$. %
\end{prop}

The following lemma provides an anologue of \Cref{lem6} for the specific polyhedron $S(x)$, instead of general subspaces.
\begin{lem}
    \label{lemma_ConformalVectors}
    Let $x \in P$ and $2 \leq q < n$. Any $s \in S(x)$ can be written as
    \begin{align}
        s = \sum_{J \in \mathcal{J}_q} s_J,
    \end{align}
    where each $s_J \in S_J(x)$, contains at most two nonzero entries, and conformal to $s$.
\end{lem}
\begin{proof}
    Let $x \in P$, $s \in S(x)$ and
    \begin{align}
        K = \{ u \in \mathbb{R}^n \, : \, a^\top u = 0 \}.
    \end{align}
    Vector $s$ is contained in subspace $K$ and thus, by \Cref{lem6}, it follows that $s$ can be written as {$s = \sum_{t = 1}^r s^t$}, where each $s^t$ is an elementary vector of $K$ and conformal to $s$. By the conformality property, the entries of these vectors satisfy
    \begin{align}
        &0 \leq s^t_i \leq s_i \quad \forall i \text{ such that } s_i \geq 0, \\ &s_i \leq s^t_i \leq 0 \quad \forall i \text{ such that } s_i \leq 0. 
    \end{align}
    As $s \in S(x)$, its entries satisfy
    \begin{align}
        e_i - x_i \leq s_i \leq g_i - x_i.
    \end{align}
    Since $x \in P$ (and thus, $e \leq x \leq g$), we have $e_i - x_i \leq 0 \leq g_i - x_i$. Therefore,
    \begin{align}
    \label{eqn_sTinclusion}
    \left.
        \begin{array}{ll}
        &0 \leq s^t_i \leq s_i \leq g_i - x_i \quad \forall i \text{ such that } s_i \geq 0, 
        \\ &e_i - x_i \leq s_i \leq s^t_i \leq 0 \quad \forall i \text{ such that } s_i \leq 0. 
        \end{array}
        \right\} \Longrightarrow e - x \leq s^t \leq g - x \, \, \forall t.
    \end{align}
    Moreover, $s^t \in K$, which implies $a^\top s^t = 0$. Combined with \eqref{eqn_sTinclusion}, this shows that $s^t \in S(x)$, see \eqref{setS}. Furthermore, each $s^t$ is an elementary vector of $K$, which implies that $s^t$ has at most two nonzero entries. This shows that
    \begin{align}
        \label{eqn_sInSj}
        s^t \in S_J(x), \text{ for some } J \in \mathcal{J}_2,
    \end{align}
    which proves the lemma for $q = 2$. For any $q > 2$, the lemma follows from noting that $| \mathcal{J}_q | > | \mathcal{J}_2 |$, which proves the existence of injective function $h: \mathcal{J}_2 \to \mathcal{J}_q$. Then, \eqref{eqn_sInSj} can be transformed in
    \begin{align}
        s^t \in S_{h(J)}(x), \text{ for some } h(J) \in \mathcal{J}_q.
    \end{align}
\end{proof}
We remark that, for any 2 vectors $s_J$ and $s_{J^\prime}$ appearing in the decomposition of \Cref{lemma_ConformalVectors}, we have
\begin{align}
    \label{eqn_conformalProperty}
    s_{J}^\top s_{J^\prime}^{\phantom{\top}} \geq 0,    
\end{align}
which follows from the conformality property.

In  the convergence analysis, we use properties of separable functions. We say that a convex function $h:\mathbb{R}^n \rightarrow \mathbb{R}$ is {\em additively  separable} if 
$$h(x) = \sum_{j=1}^{n} h_j(x_j),~~~\forall x\in \mathbb{R}^n,$$
where $h_j:\mathbb{R}\rightarrow \mathbb{R}$,  $j\in [n]$ are convex functions.  For example, the function $\|x\|^2$ is  additively separable   since
  $ \|x\|^2 = \sum_{j=1}^{n}  x_j^2$.

\begin{lem}\cite[Lemma 6.1]{tseng2009} \label{lem3.5}
 Let $h$ be a convex, additively  separable function.  For any $x$, $x+d$ from the domain of $h$, let $d$ be expressed as $d = \sum_{t=1}^{r} d^t$, for some $r\geq 1$   and  nonzero $d^t \in \mathbb{R}^n$  conformal to $d$ for all $t \in [r]$.
 Then
  $$ h(x+d) -h(x) \geq \sum_{t=1}^{r}\left ( h(x+ d^t)-h(x) \right ).$$
\end{lem}
Note that vectors $d^t$, $t \in [r]$ in \Cref{lem3.5} do not have to be elementary. We need the following result in the proof of \Cref{Lemma_expected_value}.
\begin{lem}\label{lem3.6}
Let $L_J$ be defined  as in  \Cref{assump2}, $S(x)$ (resp.,~$S_J(x)$) defined as in \eqref{setS} (resp.,~\eqref{setSJ}), and $x\in P$, see~\eqref{feasibleP}. Assume that $s \in S(x)$ is expressed as $s=\sum_{J\in \mathcal{J}}  s_J$ where $s_J \in S_J(x)$  is conformal to $s$. 
 Then 
  $$\sum_{J\in \mathcal{J}} L_J  \|s_J\|^2 \leq \| \sum_{J\in \mathcal{J}} \sqrt{L_J}  s_J \|^2.$$
\end{lem}
\begin{proof}
Define $h(x):=\|x\|^2$. From \Cref{lem3.5},  we have
\[
  \big\|\displaystyle\sum_{J\in \mathcal{J}}   \sqrt{L_J}  s_J\big\|^2 \geq \sum_{J\in \mathcal{J}} \big[ \|0 +  \sqrt{L_J}  s_J\|^2 - \|0\|^2\big] = \sum_{J\in \mathcal{J}} L_J \| s_J\|^2.
\]
Formally, \Cref{lem3.5} requires all the terms $\sqrt{L_J} s_J$ appearing in the decomposition to be nonzero, which might not be the case here. To resolve this, one can sum over a subset of $\mathcal{J}$ given by $\{ J \in \mathcal{J} \, | \, s_J \neq 0 \}$.
\end{proof}
Let us introduce the set 
$$\xi^k = \{ J_0, J_1, \ldots, J_k\},$$
where $J_i \in \mathcal{J}_q$ ($i=0,1, \ldots,k$) are random selections obtained from \Cref{Alg1}.

\begin{prop}\label{Lemma_expected_value}
Let $\{x^{m}\}$ be a sequence generated by \Cref{Alg1}  using the uniform distribution.  
   Then the following inequality holds for all~$m$:
\begin{equation}\label{expected_value}
  E \left[
  \psi_{L_{J_m} \textbf{1}}(d_{J_m}(x^m); x^m) | \xi^{m-1}   \right] \leq
  \left(  1-\frac{1}{z} \right) f(x^{m}) + \frac{1}{z}
   \psi_{\mu \Gamma}(d_{\mu \Gamma}(x^m); x^m),
     \end{equation}
where $z=\binom{n}{q}$, $\mu =\binom{n-1}{q-1}$ and $\psi_{\alpha\Gamma}(\cdot~;x^m)$ is defined in \eqref{definition-of-psi}.
\end{prop}

\begin{proof}
(See also  \cite[Lemma 8]{patrascu2015efficient}.)
For the sake of brevity, we omit $m$ and $q$ from the notations. Therefore,  the current  point is denoted by  $x$, $L_{J_m}$ is replaced by $L_{J}$,     $\xi^{ m-1}$ by $\xi$, and 
$\mathcal{J}_q$ by $\mathcal{J}$. 

From  the definition of 
$\psi_{\alpha \Gamma} (\cdot~; x)$, see  \eqref{definition-of-psi}, and the definition of 
$d_{J}$, see \eqref{update-step}, it follows:
  \begin{equation}\label{eq16}
    \psi_{L_J \textbf{1}} (d_{J}(x);x) \leq   \psi_{L_J \textbf{1}} (s_J;x), ~\forall s_J \in S_J(x).
  \end{equation}
  Now, by taking expectation on both sides  of \eqref{eq16} w.r.t.~random variable $J\in \mathcal{J}$ (see \eqref{eqn_Jdef}) conditioned on $\xi$,   it follows:
\begin{align}
\label{eq3-6-12}
    E\left[\psi_{L_J \textbf{1}}(d_{J}(x); x) | \xi   \right]  {\leq} f(x) + \frac{1}{z}\left[ \sum_{J\in\mathcal{J}}
      \langle \nabla_J f (x), s_J\rangle {+ \sum_{J\in\mathcal{J}} \frac{L_J}{2} \|s_J\|^2} \right],
\end{align}
where  $s_J \in S_J(x)$, and $z$ is defined in \eqref{eqn_numberZ}. 
In particular, \eqref{eq3-6-12} holds for  $s_J \in S_J(x)$, $J\in \mathcal{J}$  such that each $s_J$ is conformal to $d_{\mu \Gamma}(x)$ and $d_{\mu \Gamma}(x) = \sum_{J \in \mathcal{J}} s_J$, see \eqref{definition-of-maximizer-for-psi}. Note that such decomposition exists by \Cref{lemma_ConformalVectors}. 

From the definition of $\Gamma_i$ \eqref{definition-of-gamma_i}, it follows that
\begin{align}
    \label{eqn_Lj_ineq}
    L_J \leq \mu \min_{i \in J}  \{ \Gamma_i \}.
\end{align}
Combining the definition of $D_{\Gamma}$ \eqref{defintionofD}, inequalities \eqref{eqn_conformalProperty} and \eqref{eqn_Lj_ineq}, and \Cref{lem3.6}, we find
\begin{align}
     \sum_{J\in\mathcal{J}} L_J \| s_J \|^2 =  \sum_{J\in\mathcal{J}}  \| \sqrt{L_J} s_J \|^2 \leq  \| \sum_{J\in\mathcal{J}}   \sqrt{L_J} s_J \|^2 \leq  \| \sum_{J\in\mathcal{J}}  D_{\mu \Gamma}^{1/2} s_J \|^2 = \mu \| d_{\mu \Gamma}(x) \|_{\Gamma}^2.
\end{align}
Substituting this inequality in \eqref{eq3-6-12} yields
\begin{align}
    E\left[\psi_{L_J \textbf{1}}(d_{J}(x); x) | \xi   \right]  & {\leq }   f(x) + \frac{1}{z}\bigg[  \langle \nabla f (x), d_{\mu \Gamma}(x) \rangle +
    \frac{\mu}{2} \| d_{\mu \Gamma}(x)\|_{\Gamma}^2  \bigg] \\
    &= \Big( 1-\frac{1}{z} \Big) f(x) + \frac{1}{z}\left [ \psi_{\mu \Gamma}(d_{\mu \Gamma}(x); x) \right].
\end{align}
\end{proof}

The following lemma  shows that the optimal solution of the convex optimization subproblem  \eqref{update_subproblem} provides  an descending direction for the non-convex minimization problem~\eqref{general}.

\begin{lem}\label{lemma3-6}
Let $\{x^{m}\}$ be a sequence generated by \Cref{Alg1}.
For the update step $d_{J_m}$, defined in \eqref{update-step}, it holds that
\begin{equation}
  f(x^m+d_{J_m}) {\leq f(x^m)}.
\end{equation}
\end{lem}
\begin{proof}
For the sake of brevity, we omit $m$ from the notations. Therefore, the current  point is denoted by  $x$, the update step is $d_J$, $J_m$ is replaced by  $J$ and $L_{J_m}$ by $L_{J}$.

Let $\gamma, \delta \in \mathbb{R}_+^q$
satisfy \eqref{update_subproblem-second_optimality condition} and \eqref{update_subproblem-third_optimality condition}. We rewrite \eqref{update_subproblem-first_optimality condition}
in a vector notation 
as follows:
\begin{equation}\label{eq3-7-1}
     L_J d_J + \nabla_Jf(x) +\lambda a_J -\gamma +\delta =0,
\end{equation}
where  $d_J$ (resp.,~$a_J$), is the projection of $d$ (resp.,~$a$) onto the subspace $\mathbb{R}^q$ identified by $J$. Inner vector product of \eqref{eq3-7-1} with $d_J$ leads to
\[ 
L_J \langle  d_J,d_J\rangle + \langle \nabla_Jf(x),d_J\rangle +\lambda \langle a_J,d_J\rangle   -\langle \gamma,d_J\rangle +\langle \delta, d_J\rangle = 0. \label{eq3-10}
\] 
Considering \eqref{eq3-1-1} we have that $\langle a_J,d_J\rangle = 0$.
 From the equations  \eqref{update_subproblem-second_optimality condition} and \eqref{update_subproblem-third_optimality condition} we obtain
\begin{eqnarray*}
 \langle \gamma, d_J\rangle = \langle \gamma, u_J-x_J\rangle  &=& \langle \gamma , u_J - e_J \rangle - \langle \gamma, x_J - e_J \rangle= - \langle \gamma, x_J - e_J \rangle, \\
 \langle \delta , d_J\rangle  = \langle \delta, u_J-x_J\rangle &=& \langle \delta  , u_J-g_J \rangle + \langle \delta ,g_J-x_J\rangle
  =\langle \delta , g_J-x_J\rangle,
\end{eqnarray*}
 respectively.  Therefore,
   \begin{equation}\label{eq3-11}
   L_J \|d_J\|^2 + \langle \nabla_Jf(x),d_J\rangle  + \langle \gamma,x_J - e_J\rangle +\langle \delta, g_J-x_J \rangle =0.
   \end{equation}
Lastly, the substitution of \eqref{eq3-11} in \eqref{identical-final-result-of-q-coordinate}, with $d_{J}$  instead of $s_J$,  leads to 
\begin{eqnarray}
\label{eq7-4}
  f(x + d_{J}) 
  & { \leq } &  f(x) - \frac{L_{J}}{2} \| d_{J} \|^2  - \langle \gamma,x_J - e_{J}\rangle -\langle \delta, g_{J}-x_J \rangle \\
  &{ \leq}&  f(x), \nonumber
\end{eqnarray}
which completes the proof.
\end{proof}

\Cref{lemma3-6} shows that the objective value of \eqref{general} in the next iteration of \Cref{Alg1} is not larger than the objective value in the current iteration.

For the proof of convergence, we need the following well-known property of supermartingales.
\begin{thm} \label{lem2} 
(See, e.g., \cite[Theorem 11.5]{williams1991probability})
 Let $(X)_{n\geq 1}$ be a supermartingale such that   {$\sup_n E\Big[ |X_n| \Big] <+\infty$}.  Then $\lim_{n \rightarrow \infty}  X_n = X^*$ exists almost surely (a.s.). 
  \end{thm}

The following two theorems present  convergence properties of \Cref{Alg1}.
\begin{thm}
\label{ThmM_congergence}
  Let the  objective function $f(x)$ of problem~\eqref{general} satisfy  \Cref{assump2}, and let the sequence $\{x^{m}\}$ be generated by \Cref{Alg1}  using the uniform distribution. Then, the sequence $\{f(x^m)\}$ converges to a random variable $\bar{f}$ a.s., and
 the sequence of random variables $\{ M_q(x^m, {\Gamma})\}$ converges to 0 almost surely.
\end{thm}
\begin{proof}
We have, 
\begin{equation}\label{eq17}
  f(x^{m+1}) {\leq }  f(x^m) - \frac{L_{J_m}}{2}\|d_{J_m}(x^m) \|^2-\langle \gamma,x^m_{J}-e_{J_m}\rangle
  - \langle \delta, g_{J_m} -x^m_{J_m} \rangle,
\end{equation}
$\forall m\geq 0$,
see~\eqref{eq7-4}. This inequality shows that the objective function decreases in each step of  \Cref{Alg1} meaning that the algorithm is descending.
Applying expectation conditioned on $\xi^{m-1}$  
we obtain
\begin{align}\label{eq17-2}
         E\bigg[f(x^{m+1})|\xi^{m-1}\bigg]  {\leq } E\bigg[f(x^m)|\xi^{m-1}\bigg] 
     - E\bigg[\frac{L_{J_m}}{2}\|d_{J_m}(x^m) \|^2+ \langle \gamma,x^m_J-e_{J_m}\rangle +\langle \delta, g_{J_m}
  -x^m_J \rangle |\xi^{m-1}\bigg].
        \end{align}

Considering the fact that {$E[ \big|f(x^{m+1}) \big| \, | \, \xi^{m-1}]<\infty$}, since $\{x^m\}$ is in the bounded set $P$ and $f$ is differentiable, the martingale convergence \Cref{lem2}  states that  $\{f(x^m)\}$ converges to a random variable $\bar{f}$ a.s.~when $m\rightarrow \infty$. 
Due to the almost sure convergence of the sequence $\{f(x^m)\}$, it can be immediately observed  that
$$\lim_{m\rightarrow \infty} \left (  f(x^m ) - f(x^{m+1}) \right ) = 0 \mbox{  a.s}.$$
Considering this, it follows  from \eqref{eq17} that
\begin{equation}
  { \lim_{m\rightarrow \infty}}\langle \gamma,x^m_J-e_{J_m} \rangle    +\langle \delta, g_{J_m} -x^m_J \rangle
   \leq {\lim_{m\rightarrow \infty}} - \frac{L_{J_m}}{2}  \| d_{J_m}(x^m)\|^2\mbox{  a.s.}
\end{equation}
The fact that $\langle \gamma,x^m_J-e_{J_m}\rangle    +\langle \delta,  g_{J_m} -x^m_J \rangle $ is nonnegative, and the right hand side is nonpositive implies
\begin{equation} \label{AlmostSure}
 \lim_{m\rightarrow \infty} \langle \gamma,x^m_J-e_{J_m} \rangle    +\langle \delta, g_{J_m} - x^m_J \rangle = 0,  \lim_{m\rightarrow \infty}\| d_{J_m}(x^m)\|=\|x^{m+1} -x^m\| =0 \mbox{ a.s.}
\end{equation}

From \eqref{AlmostSure}, we also conclude that
$$E\bigg[\frac{L_{J_m}}{2}\|d_{J_m}(x^m) \|^2+ \langle \gamma,x^m_J-e_{J_m} \rangle +\langle \delta, g_{J_m} 
  -x^m_J \rangle |\xi^{m-1}\bigg] \rightarrow 0, \mbox{ a.s.}$$

Let us now prove the second statement. From \Cref{Lemma_expected_value}, we obtain a sequence which bounds $\psi_{\mu \Gamma} (d_{\mu \Gamma}; x)$ from below as follows:
\begin{equation}\nonumber
 z E\left[\psi_{L_{J_m} \textbf{1}}(d_{J_m}(x^m); x^m) | \xi^{m-1}   \right] - (z-1) f(x^m) {\leq} \psi_{\mu \Gamma} (d_{\mu\Gamma}(x^m); x^m).
\end{equation}
Further,  it follows from \Cref{lemma_ConformalVectors} that any $s \in S(x)$ has a conformal realization
given as $s = \sum_{J} s_J$, where $s_J \in S_J(x)$ are conformal to  $s$ and have at most two nonzero
elements. By exploiting this and Jensen's inequality,  we derive an upper bound for $\psi_{\mu \Gamma} (d_{\mu \Gamma}(x^m); x^m) $ as follows:
\begin{align}
   \psi_{\mu \Gamma} (d_{\mu \Gamma}(x^m); x^m) &=\displaystyle {\min_{s\in S(x^m)}}  \left[ f (x^m) + \langle \nabla f (x^m), s\rangle
   {+ \frac{1}{2} \|s\|^2_{\mu \Gamma}} \right ] \\
   &= \displaystyle {\min_{s\in S(x^m)}} \left [ f (x^m) + \big\langle \nabla f (x^m), \sum_{J}s_J\big\rangle
   {+ \frac{1}{2} \bigg  \|\sum_{J}s_J\bigg\|^2_{\mu \Gamma}} \right]\\
   &=\displaystyle {\min_{\tilde{s}_J\in S_J(x^m)}} \left [ f (x^m) + \frac{1}{z} \big\langle \nabla f (x^m),
   \sum_{J}\tilde{s}_J\big \rangle {+ \frac{1}{2} \bigg  \|\frac{1}{z} \sum_{J}\tilde{s}_J\bigg\|^2_{\mu \Gamma}} \right ] \\
   & {\leq  \displaystyle \min_{\tilde{s}_J\in S_J(x^m)} }\left [ f (x^m) +\frac{1}{z} \sum_{J}\langle \nabla f (x^m), \tilde{s}_J \rangle
 {+ \frac{1}{2z}  \sum_{J} \big\| \tilde{s}_J\big\|^2_{\mu \Gamma}} \right ]\\
 &= E \left [\psi_{\mu \Gamma} (d_{J_m}(x^m); x^m)| \xi^{m-1} \right ],
\end{align}
where $\tilde{s}_J = zs_J$, see \eqref{eqn_numberZ}. 
 
Let us  summarize the previous results. In particular, bounds on $\psi_{\mu \Gamma} (d_{\mu \Gamma}(x^m); x^{m})$ are given below
 \begin{equation}
 \label{uper-lower}
  \begin{array}{rl}
    z E\left[\psi_{L_{J_m} \textbf{1}}(d_{J_m}(x^m); x^m) | \xi^{m-1}   \right] - (z-1) f(x^m) & \leq   \psi_{\mu \Gamma} (d_{\mu \Gamma}(x^m); x^{m}) \\[1ex]
    & \leq  E[\psi_{\mu \Gamma} (d_{J_m}(x^m); x^m)| \xi^{m-1}] 
  \end{array}
   \end{equation}
Recall that $d_{J_m}(x^m) =u_J^m -x_J^m\rightarrow 0$ a.s. Therefore
$E[\psi_{\mu \Gamma} (d_{J_m}(x^m); x^m)|\xi^{m-1}]$ converges to $\bar{f}$ a.s.~when $m\rightarrow\infty$.  Clearly,  sequences of
lower and upper bounds of $\psi_{\mu \Gamma} (d_{\mu \Gamma}(x^m); x^m)$  converge to $\bar{f}$ and therefore  $\psi_{\mu \Gamma} (d_{\mu \Gamma}(x^m); x^m)$ converges to $\bar{f}$ a.s.~for $m\rightarrow\infty$.

Note that the function $\psi_{\mu \Gamma} (s; x)$ is strongly convex in the variable $s$   with parameter $\mu$  w.r.t.~norm $\|\cdot\|_{\Gamma}$, see \Cref{stronglyConvex}. This follows from {$\nabla^2\psi_{\mu \Gamma} = \mu D_{\Gamma} \succ 0$}.
Therefore, $d_{\mu \Gamma}(x)$ is the unique minimizer and the following inequality holds: 
\begin{equation}\label{eq18}
   {\psi_{\mu \Gamma} (s; x)  ~{ \geq }~ \psi_{\mu \Gamma} (d_{\mu \Gamma}(x); x) {+ \nabla \psi_{\mu \Gamma}(d_{\mu \Gamma}(x);x)^\top (s - d_{\mu \Gamma}(x) )+ \frac{\mu}{2} \| s-d_{\mu \Gamma}(x)\|_{\Gamma}^2,
   ~~~~\forall x,s\in \mathbb{R}^n}},
 \end{equation}
which leads to
\begin{align}
   \psi_{\mu \Gamma} (0; x^m)   &{\geq} \psi_{\mu \Gamma} (d_{\mu \Gamma}(x^m); x^m) {- \nabla \psi_{\mu \Gamma}(d_{\mu \Gamma}(x^m); x^m )^\top d_{\mu \Gamma}(x^m) 
   + \frac{\mu }{2} \| d_{\mu \Gamma}(x^m) \|_{\Gamma}^2} \\
& {\geq \psi_{\mu \Gamma} (d_{\mu \Gamma}(x^m); x^m) + \frac{\mu}{2} \| d_{\mu \Gamma}(x^m) \|_{\Gamma}^2}, \label{zeroAlmostsure}
\end{align}
where the last inequality follows from {$\nabla \psi_{\mu \Gamma}(d_{\mu \Gamma}(x^m); x^m)^\top d_{\mu \Gamma}(x^m) \leq 0$}, see \Cref{lem_inProdIneq}. Since $ \psi_{\mu \Gamma} (0; x^m)= f(x^m)$, it follows that $\psi_{\mu \Gamma} (0; x^m) $ converges to $\bar{f}$ a.s.~when $m\rightarrow\infty$.
Thus, both sequences
$\psi_{\mu \Gamma} (0; x^m) $ and $\psi_{\mu \Gamma} (d_{\mu \Gamma}(x^m); x^m)$, converge to $\bar{f}$ a.s. for $m\rightarrow\infty$.

From the previous discussion and \eqref{zeroAlmostsure} it follows that
\begin{align}
{\bar{f}  \geq \bar{f}
   + \frac{\mu }{2} \| d_{\mu \Gamma}(x^m) \|_{\Gamma}^2 ~~~\mbox{a.s.}},
\end{align}
which results in 
  $ \| d_{\mu \Gamma}(x^m ) \|_{\Gamma}^2 \leq 0,$
from where it follows that  the sequence {$M_q(x^m , \Gamma) = \|d_{\mu \Gamma}(x^m )\|_{\Gamma}$} converges to 0 a.s. when $m \to \infty$.
\end{proof}
We use the result of \Cref{ThmM_congergence} to prove the following theorem.
\begin{thm}\label{thm:accumulation point}
Assume that  the  objective function $f(x)$ of problem~\eqref{general} satisfies \Cref{assump2}, and let the sequence $\{x^{m}\}$ be generated by \Cref{Alg1}  using the uniform distribution.
Then  any accumulation point of the sequence $\{x^m\}$ is a stationary point for~\eqref{general}.
\end{thm}
\begin{proof}
Assume that the entire sequence $\{x^{m}\}$  is convergent, and let $\overline{x}$ be the limit point of this sequence. Using the fact that $d_{\mu \Gamma}(x^{m})$, see~\eqref{definition-of-maximizer-for-psi}, is the minimizer of $\psi_{\mu \Gamma}(\cdot~; x^m)$, see~\eqref{definition-of-psi}, we have 
 \begin{align}
 \psi_{\mu\Gamma}(d_{\mu\Gamma}(x^m);x^m)     &=
  f(x^m) + \left\langle \nabla f(x^m) ,d_{\mu \Gamma}(x^{m})  \right\rangle
  {+ \frac{\mu}{2} \| d_{\mu \Gamma}(x^{m})\|_{\Gamma}^2  \leq } f(x^m) +  \left\langle \nabla f(x^m) ,s  \right\rangle {+ \frac{\mu}{2} \|s\|_{\Gamma}^2}  
 \end{align}
for all $s\in S(x^m)$. Taking the limit $m\rightarrow \infty$, we obtain
\begin{eqnarray}
  f(\overline{x}) + \left\langle \nabla f(\overline{x}) ,d_{\mu\Gamma}(\overline{x})  \right\rangle
  {+ \frac{\mu }{2} \| d_{\mu\Gamma}(\overline{x})
  \|_{\Gamma}^2      
     \leq} f(\overline{x}) +
  \left\langle \nabla f(\overline{x}) ,s  \right\rangle {+ \frac{\mu }{2} \|s\|_{\Gamma}^2},
\end{eqnarray}
for all $s\in S(\overline{x})$.
Thus, {$\psi_{\mu \Gamma} (d_{\mu \Gamma}(\overline{x});\overline{x}) \leq \psi_{\mu \Gamma} (s;\overline{x})$} for all $s\in S(\overline{x})$.
This inequality  accompanied with  the a.s.~convergence of $\{ d_{\mu \Gamma}({x}^m) \}$ to zero (see \Cref{ThmM_congergence}) implies that
$d_{\mu \Gamma}(\overline{x}) =0$ is the {minimum} of   \eqref{definition-of-psi} for $x= \overline{x}$  and thus $M_q (\overline{x}, \Gamma)=0$.
 From \Cref{lem3.1},  it follows that $\overline{x}$ is  a stationary point of non-convex optimization problem~\eqref{general}.
\end{proof}
The following theorem provides the convergence rate for \Cref{Alg1}.
\begin{thm}\label{thm:convergenceRate}
 Let the  objective function $f(x)$ of problem~\eqref{general} satisfy \Cref{assump2}.
 Then, \Cref{Alg1}, based on the uniform distribution, generates a sequence $\{x^m\}$ satisfying the following convergence rate for the expected values of the { stationarity measure}:
\begin{equation}\label{rate}
  \min_{0\leq k \leq m} E\left[  (M_q(x^k, \Gamma))^2 \right] \leq \frac{2n {\left ( f(x^0)-f^* \right )}}{q(m+1)},~~~ \forall m\geq 0,
\end{equation}
where $f^*$ is the optimal value for non-convex optimization problem~\eqref{general}, $M_q(x, \Gamma)$ is given in  \eqref{optimality_measure}, and
$\Gamma$  in \eqref{definition-of-gamma_i}.
\end{thm}
\begin{proof}
For the sake of brevity, we omit $m$ from the notations. Therefore,  the current  point is denoted by  $x$, $L_{J_m}$ is replaced by $L_{J}$, and    $\xi^{ m-1}$ by $\xi$. 

Considering \eqref{identical-final-result-of-q-coordinate} for $d_J$ yields:
 \begin{align}
  f(x+d_J) {\leq f(x) + \langle \nabla_J f(x), d_J  \rangle  +  \frac{L_J}{2} \|d_J\|^2.}
\end{align}
Let $x^+:=x+d_J$ and 
take expectation conditioned on $\xi$ for the above inequality:
\begin{align}
  E\big[f(x^+)|\xi\big] \leq  E\Big[f(x) + \langle \nabla_J f(x), d_J  \rangle  {+  \frac{L_J}{2} \|d_J \|^2 } |\xi\Big] = {E\Big[\psi_{L_J \textbf{1}}(d_{J}(x); x) |\xi \Big]}.
  \end{align}
Combining the previous inequality and \Cref{Lemma_expected_value},  we obtain
\begin{equation}\label{eq20}
  E\left[  f(x^+) | \xi   \right] {\leq} \left(  1-\frac{1}{z} \right) f(x) + \frac{1}{z} \psi_{\mu \Gamma}(d_{\mu \Gamma}(x); x),
\end{equation}
from where it follows:
\begin{equation}\label{eq20-2}
   E\left[  f(x^+) | \xi   \right]  \leq \left(  1-\frac{1}{z} \right)  E\left[  f(x) | \xi   \right]  + \frac{1}{z} E\left[ \psi_{\mu \Gamma}(d_{\mu \Gamma}(x); x)|\xi \right].
  \end{equation}
{Thus, 
\begin{align}
E\left[  f(x) | \xi   \right] - E\left[  f(x^+) | \xi   \right] \, \, &\geq \, \, \frac{1}{z} \Big(E\left[  f(x) | \xi   \right] -E[\psi_{\mu \Gamma}(d_{\mu \Gamma}(x); x)|\xi ] \Big)  \\
   &= \, \, \frac{1}{z} \Big(E[ \psi_{\mu \Gamma}(0; x)|\xi  ] - E[\psi_{\mu \Gamma}(d_{\mu \Gamma}(x); x)|\xi ]\Big)\\
      &\geq \, \,  \frac{\mu}{2z} E \Big[  \big\|d_{\mu \Gamma}(x) \big\|_{\Gamma}^2 \Big] = \frac{q}{2n} E \Big[  \big\|d_{\mu \Gamma}(x) \big\|_{\Gamma}^2 \Big].
   \end{align}}
The last inequality is an implication of $\mu$-strongly convexity of {$\psi_{\mu \Gamma}(d_{\mu \Gamma}(x); x)$}, see \eqref{zeroAlmostsure}.
Getting back to the notation dependent  on iterations $m$ and 
    summing up w.r.t.~entire history, we obtain:
   \begin{equation}
      \frac{q}{2n}
      \sum_{k=0}^{m} E\big[ (M_q(x^k, \Gamma))^2\big]\leq   {f(x^0)-f^*},
   \end{equation}
 from where the theorem follows.
\end{proof}
\begin{rem}
\Cref{thm:convergenceRate} shows that for  $q=2$, the convergence rate of \Cref{Alg1} coincides  with the convergence of Algorithm 2-RCD \cite[Theorem 5]{patrascu2015efficient}. However, for $q>2 $ our algorithm converges faster than the algorithm from~\cite{patrascu2015efficient}. Note that the algorithm from \cite{patrascu2015efficient} allows for choosing blocks of coupled variables. For this comparison, we assume that these blocks only contain single variables.
\end{rem}

\section{Numerical experiments}
\label{sect:numerics}
In this section we present numerical results for solving  the densest $k$-subgraph  problem  and the eigenvalue complementarity   problem. Numerical results are performed on an \mbox{Intel i7-1165G7} 2.80GHz processor with 4 cores, and 16GB RAM.

{
In the following numerical experiments, we verify whether the computed points $x^*$ are stationary. Theoretically, by using \Cref{lem3.1}, this can be done by checking whether $M_q(x^*,\Gamma) = \| d_{\mu \Gamma}(x^*) \|= 0$. However, in practice, $n$ can be large, which prohibits the computation of $\mu = \binom{n-1}{q-1}$, and consequently, the computation of $M_q(x^*,\Gamma)$. We therefore consider another, related value.

A point $x^* \in P$, see \eqref{feasibleP}, is a stationary point of $f$ if $\langle \nabla f(x^*), x-x^* \rangle \geq 0$ for all $x \in P$ (see e.g., \cite[Proposition 2.1.2]{bertsekas1997nonlinear}). We define the function $\widetilde{M} : P \to \mathbb{R}_{+}$, as
\begin{align}
    \label{eqn_wideTildeMDef}
    \widetilde{M}(x^*) := - \min_{x \in P} \quad  \langle \nabla f(x^*), x-x^* \rangle
\end{align}
Note that $\widetilde{M}$ can be computed by solving a linear programme on $n$ variables. Moreover, $\widetilde{M}(x^*) = 0$ implies that $x^*$ is stationary.}

One can download our codes from the following link:
\begin{center}
{\url{https://github.com/LMSinjorgo/q-RCCD_algorithm}}
\end{center}

\subsection{The densest \texorpdfstring{$k$}{k}-subgraph problem }
\label{Section_K_densest_subgraph}
Let $G$ be an undirected graph with $n$ vertices and {$k \in \mathbb{N}$} a given number such that  $3\leq k \leq n-2$. The \textit{densest $k$-subgraph problem} is the problem of finding a subgraph of $G$ with  $k$ vertices and the maximum number of edges. 
The D$k$S problem is known to be NP-hard. The problem plays a role in analyzing web graphs and social networks, but also in computational biology and cryptography. There does not exist a polynomial time approximation scheme for the D$k$S problem in general graphs~\cite{KhotSubhash}. Exact approaches for solving the D$k$S problem report solutions for instances only up to 160 vertices, see e.g.,~\cite{KrisMalRoup:16}. 
Thus, obtaining good bounds for the problem  is a crucial step for solving large scale instances. 

The D$k$S problem can be formulated as follows:
\begin{equation}\label{binary-nonconvex}
\max\left \{x^\top Ax ~:~ \sum_{i=1}^{n}x_i=k, ~x_i \in \{0,1\}, ~i\in [n] \right \},
  \end{equation}
 where $A$ is the adjacency matrix of $G$. We consider here the continuous relaxation of the D$k$S problem, i.e., 
\begin{equation}\label{relaxed_problem}
\max\left \{x^\top Ax ~:~ \sum_{i=1}^{n}x_i=k, ~0\leq x_i\leq 1, ~i\in [n] \right \},
\end{equation}
and use the \qRCCD{} algorithm to compute {stationary points} for various instances.

The next lemma shows that, in case of the D$k$S problem, one can easily compute a Lipschitz constant in \eqref{q-coordinate}.
\begin{lem}\label{lem2.2}
  Let $f(x)= x^\top Ax$, where $A$ is the adjacency matrix of a graph. Then, the Lipschitz condition \eqref{q-coordinate} is satisfied for $L_J = 2\|A_J\|$, where $A_J$ is the $|J| \times |J|$ principal submatrix of $A$, with columns and rows indexed by $J.$
\end{lem}
\begin{proof} 
Let $J \subseteq [n]$, and $\nabla_J f$ and $s_J$ be defined as in \eqref{q-coordinate}. For any $J \subseteq [n]$, we have $\nabla_J f(x) =(2Ax)_J$, and so $\nabla_J f$ is linear. Thus,
  \begin{equation}
  \label{eq21}
  \|\nabla_J f(x+s_J) - \nabla_J f(x)  \| = \| \nabla_J f(s_J) \|  {\leq  2\|A_J\| \, \|s_J\|}.
   \vspace{-1.2em}
\end{equation}
\end{proof}

\begin{rem}
    Because of norm equivalency in $\mathbb{R}^n$, it is not important to use a particular norm $\| A_J \|_p$ in the calculation. Since $A$ is the adjacency matrix of a simple graph, $A_J$ is symmetric which implies that $\|A_J\|_{\infty}= \|A_J\|_{1}  = \triangle$, where $\triangle$ denotes the largest degree in the subgraph. We note that, for large subgraphs, computing $\triangle$ is significantly cheaper than computing $\| A_J \|_2$.
\end{rem}
To compute projections onto the feasible set of \eqref{relaxed_problem}, see \eqref{eqn_closedFormUpdate}, we use the algorithm proposed in \cite{wang2015projection}, which runs in $\mathcal{O}(q^2)$. 

We first test the performance  of the \qRCCD{} algorithm on  the Erd\"os-R\'enyi graphs and  the Erd\"os-R\'enyi graphs with  planted subgraphs.  Then, we test our algorithm on graphs from the literature. In the Erd\"os-R\'enyi graph $G_{p}(n)$ with $n$ vertices, each edge is generated independently of other edges with probability $p\in ( 0,1 ]$. The graph obtained by planting a complete subgraph with $m$ vertices in $G_{p}(n)$  is known as the Erd\"os-R\'enyi graph with a planted subgraph and denoted by $P^m_p(n)$. 

In 
\Cref{table_ERsimpleBounds}, for graph $G_{0.5}(2048)$ and $k = 30$, we show the performance of the \qRCCD{} algorithm, for $q \in \{2, 50, 100, 200, 500, 750 \}$ and $500, 1000, 5000, 10000$ iterations. We report the obtained lower bound on \eqref{relaxed_problem} and the required computation time in seconds, { averaged over 30 different runs of the \qRCCD{} algorithm}.
As starting vector, we take $x=(k/n)\mathbf{1}_n$. Observe that the bounds are stronger for higher values of $q$, at the expense of greater computation times. Furthermore, smaller values of $q$ require more iterations to obtain similar bounds obtained by higher values of $q$.
{ For example, for $q=100$ and $10^4$ iterations the objective value is $724.411$, and for  $q=200$ and only $5 \cdot 10^3$ iterations the objective value is  $724.874$.  
Thus, the two cases lead to similar bounds while the number of iterations times $q$ is the same, at $10^6.$}

In {\Cref{table_plantedSG}}, we test the \qRCCD{} algorithm for graph $P^{100}_{0.3}(4096)$. For each $q \in \{ 200, 300, 400,$ $500 \}$, we run the \qRCCD{} algorithm for {$750$ and  $1000$ iterations for the $P^{100}_{0.3}(4096)$} graph. Note that $P^{100}_{0.3}(4096)$ attains an optimal objective value of 9900 with high probability. In the table, we report the minimum, median, mean and maximum determined objective values 
{ averaged over 100 different runs of the \qRCCD{} algorithm}.
{We also report the average value of $\widetilde{M}(x)$, see \eqref{eqn_wideTildeMDef}, using the $x$ computed at the final iteration, and the average time (in seconds) for each $q$ and two  iteration numbers.} Observe that with relatively little computational effort, we are able to determine near optimal bounds for $q = 500$ in $750$ iterations. {Moreover, for $q = 500$ and $1000$ iterations, the \qRCCD{} algorithm attains the optimal value of 9900 in every run, while the average value of the stationarity measure $\widetilde{M}(x)$ is $3.1$E$-06$.}

{
In \Cref{bigTable}  we compare the \qRCCD{} algorithm with three different numerical methods for solving the D$k$S problem on 12 graphs.  In this table we set $k = 200$. The methods we compare are: the ADMM, the PGM, and a version of the 2-Random Coordinate Descent (2-RCD) from \cite{patrascu2015efficient}.
The ADMM for the D$k$S problem is proposed
 by~\citeauthor{bombina2020convex}~\cite{bombina2020convex}, and we use the code provided by the authors\footnote{Code available at \url{https://github.com/pbombina/admmDSM}.}. 
The PGM is the deterministic version of 
 the \qRCCD{} algorithm, i.e., when $q = n$. The 2-RCD algorithm, with parameter ${\mathcal{B}} \in \mathbb{N}$, implemented here differs from the \qRCCD{} algorithm 
  in only  the variable selection method: given the parameter ${\mathcal{B}} \in \mathbb{N}$, we divide the $n$ variables in blocks, each of size ${\mathcal{B}}$\footnote{In case ${\mathcal{B}}$ does not divide $n$, we set ${\mathcal{B}}$ as the closest divisor of $n$.}. In every iteration, two blocks are picked, uniformly at random, and the 2$B$ variables in these blocks are updated the same way as for the \qRCCD{} algorithm. For the \qRCCD{} algorithm  and graphs with edge density of approximate order 1E$-01$, we set $q = 100$. For all other graphs we set $q = 1500$. For the 2-RCD algorithm, we chose ${\mathcal{B}} = 10$ for every graph because the preliminary results show that for this ${\mathcal{B}}$ the algorithm provides the highest objective values.

Each method runs for a set time (10, 60 or 120 seconds). We report the objective value, denoted \textbf{Obj.} We also display a lower bound on the optimal value of the D$k$S problem, denoted LB, which is obtained as follows: given the nonbinary $x \in \mathbb{R}^n$ returned by one of the four algorithms, we set the $k$ highest values in $x$ to 1, and the rest to 0. This produces a binary vector feasible to \eqref{binary-nonconvex}, and we use this feasible vector to compute a lower bound LB. Since ADMM returns a matrix $X \in \mathbb{R}^{n \times n}$, we perform the same operation on the $n$-dimensional diagonal of $X$. Additionally, for all methods other than ADMM, we report the value of $\widetilde{M}(x)$, see \eqref{eqn_wideTildeMDef}, using the nonbinary $x$ obtained at the final iteration\footnote{Due to numerical imprecision, $\widetilde{M}$ is sometimes computed as a negative number.}. For the \qRCCD{} and 2-RCD algorithms, the reported values $\textbf{Obj.}$, LB and $\widetilde{M}$ are averaged over three runs.

The ADMM algorithm requires computing a singular value decomposition (SVD) at each iteration. In \Cref{bigTable}, entries for which the SVD led to a memory error, or for which the SVD could not be computed in the allotted maximum time, are marked with *.

Seven of the 12 tested graphs in \Cref{bigTable}  are taken from \cite{snapnets,rossi2015network}, and more information on those graphs can be found in these references. For these seven graphs, we removed isolated vertices and self-loops. Among the other five graphs are three  Erd\"os-R\'enyi graphs and two Erd\"os-R\'enyi graphs with  planted subgraphs. For each graph, the number of vertices, edges, and edge-densities, are given by $|V|$, $|E|$ and $D := |E| / \binom{|V|}{2}$ respectively. 

\Cref{bigTable} demonstrates the superior performance of the \qRCCD{} algorithm for solving the D$k$S problem compared to the other algorithms. The \qRCCD{} algorithm provides the best objective values and lower bounds for most graphs. The very small values of $\widetilde{M}$, which are obtained for all non-Erd\"os-R\'enyi graphs, indicate that the \qRCCD{} algorithm converges to a stationary point. 
For the Erd\"os-R\'enyi graphs (with possible planted subgraphs), a clear evidence of stationarity might be obtained by running the \qRCCD{} algorithm for more than two minutes.}

\begin{table}[]
\centering
\caption{ {Results averaged over 30 runs on the graph $G_{0.5}(2048)$ ($k = 30$)}}
\label{table_ERsimpleBounds}
\centering
\begin{tabular}{l|llllll|c}
\hline
$q$ & 2 & 50 & 100 & 200 & 500 & 750 & \multicolumn{1}{l}{\textbf{Iter.}} \\ \hline
\textbf{Obj.} & 452.413 & 630.422 & 641.349 & 654.440 & 662.429 & 669.300 & \multirow{2}{*}{$\frac{1}{2}\cdot 10^3$} \\
Time (s) & 0.028 & 0.070 & 0.136 & 0.292 & 0.845 & 1.598 &  \\ \hline
\textbf{Obj.} & 454.781 & 676.842 & 689.236 & 700.214 & 704.653 & 710.191 & \multirow{2}{*}{$10^3$} \\
Time (s) & 0.055 & 0.117 & 0.215 & 0.475 & 1.434 & 2.836 &  \\ \hline
\textbf{Obj.} & 464.721 & 713.992 & 720.316 & 724.874 & 728.748 & 728.534 & \multirow{2}{*}{$5 \cdot 10^3$} \\
Time (s) & 0.112 & 0.299 & 0.635 & 1.605 & 5.749 & 12.093 &  \\ \hline
\textbf{Obj.} & 469.149 & 721.579 & 724.411 & 729.379 & 732.162 & 731.733 & \multirow{2}{*}{$10^4$} \\
Time (s) & 0.162 & 0.494 & 1.106 & 2.941 & 10.913 & 23.371 &  \\ \hline
\end{tabular}

\end{table}

\begin{table}[]
\centering
\caption{Results averaged over 100 runs on the graph $P^{100}_{0.3}(4096)$ ($k = 100$)}
\label{table_plantedSG}
\begin{tabular}{cl|llll|c}
\hline
\multicolumn{1}{l}{}                  & $q$           & 200      & 300      & 400      & 500      & \multicolumn{1}{l}{\textbf{Iter.}} \\ \hline
\multirow{4}{*}{\textbf{Obj.}}        & Min.          & 9849.015 & 9886.895 & 9890.530 & 9899.955 & \multirow{6}{*}{750}               \\
                                      & Median        & 9888.472 & 9899.841 & 9899.998 & 9900.000 &                                    \\
                                      & Mean          & 9884.431 & 9898.664 & 9899.803 & 9899.999 &                                    \\
                                      & Max.          & 9899.932 & 9899.994 & 9900.000 & 9900.000 &                                    \\ \cline{1-6}
\multicolumn{2}{l|}{\textbf{Avg. }   $\widetilde{M}$} & 1.6E+01 & 1.3E+00 & 2.0E$-$01 & 
7.1E$-$04 &   \Tstrut                                 \\
\multicolumn{2}{l|}{\textbf{Avg. T. (s)}}             & 1.119    & 1.902    & 2.895    & 4.087    &                                    \\ \hline
\multirow{4}{*}{\textbf{Obj.}}        & Min.          & 9853.434 & 9889.105 & 9892.962 & 9900.000 & \multirow{6}{*}{1000}              \\
                                      & Median        & 9896.287 & 9899.971 & 9900.000 & 9900.000 &                                    \\
                                      & Mean          & 9892.447 & 9899.697 & 9899.927 & 9900.000 &                                    \\
                                      & Max.          & 9899.963 & 9900.000 & 9900.000 & 9900.000 &                                    \\ \cline{1-6}
\multicolumn{2}{l|}{\textbf{Avg. }   $\widetilde{M}$} & 7.6E+00 & 3.0E$-$01 & 7.3E$-$02 & 3.1E$-$06 &       \Tstrut                             \\
\multicolumn{2}{l|}{\textbf{Avg. T. (s)}}             & 1.491    & 2.524    & 3.835    & 5.420    &                                    \\ \hline
\end{tabular}
\end{table}

\clearpage

\label{page_bigTable}
\newgeometry{left=0.4cm,bottom=2cm,top=0.1cm,right=0.4cm}

\begin{landscape}

\begin{table}[]
\footnotesize
\caption{Performance of different methods for solving the D$k$S problem ($k = 200$)}
\label{bigTable}
\centering
\begin{tabular}{llll|ll|lll|llll|lll|l} 
\hline
\multirow{2}{*}{$G$}                    & \multirow{2}{*}{$|V|$} & \multirow{2}{*}{$|E|$}    & \multirow{2}{*}{$D$}     & \multicolumn{2}{c|}{ADMM}               & \multicolumn{3}{c|}{PGM}                  & \multicolumn{4}{c|}{$q$-RCCD}                               & \multicolumn{3}{c|}{2-RCD (${\mathcal{B}}$ = 10)}                        & \multicolumn{1}{c}{\multirow{2}{*}{\textbf{T. (s)}}}  \Tstrut \\ 
\cline{5-16}
                                        &                        &                           &                          & \textbf{Obj.} & LB                      & \textbf{Obj.}   & LB    & $\widetilde{M}$ & $ q$ & \textbf{Obj.}    & LB      & $\widetilde{M}$ & \textbf{Obj.}   & LB              & $\widetilde{M}$ & \multicolumn{1}{c}{} \Tstrut                                  \\ 
\hline
\multirow{3}{*}{C1000-9}                & \multirow{3}{*}{1000}  & \multirow{3}{*}{450079}   & \multirow{3}{*}{9.0E-01} & 36023.5       & 36920                   & 37898.5         & 37896 & 7.2E-02         &      & \textbf{37911.7} & 37910.0 & 7.3E-05         & 37890.8         & 37888.7         & 7.3E-09         & 10                                                    \\
                                        &                        &                           &                          & 36055.6       & 36928                   & 37898.5         & 37896 & 5.0E-10         & 100  & \textbf{37911.7} & 37910.0 & 4.0E-09         & 37890.8         & 37888.7         & -2.4E-09        & 60                                                    \\
                                        &                        &                           &                          & 36062.1       & 36926                   & 37898.5         & 37896 & 1.4E-09         &      & \textbf{37911.7} & 37910.0 & 4.8E-09         & 37890.8         & 37888.7         & -2.9E-09        & 120                                                   \\ 
\hline
\multirow{3}{*}{CA-AstroPh}             & \multirow{3}{*}{18771} & \multirow{3}{*}{198050}   & \multirow{3}{*}{1.1E-03} & \multicolumn{2}{c|}{\multirow{3}{*}{*}} & 8901.1          & 10238 & 1.8E+03         &      & \textbf{12099.5} & 12099.3 & 4.1E-01         & 11579.3         & 11586.7         & 1.0E+02         & 10                                                    \\
                                        &                        &                           &                          & \multicolumn{2}{c|}{}                   & 11835.7         & 11866 & 7.9E+01         & 1500 & \textbf{12110.0} & 12110.0 & 1.6E+00         & 11729.3         & 11730.0         & 7.6E+00         & 60                                                    \\
                                        &                        &                           &                          & \multicolumn{2}{c|}{}                   & 11925.3         & 11928 & 9.6E+00         &      & \textbf{12112.2} & 12112.0 & 1.2E-10         & 11752.5         & 11752.7         & 3.4E+00         & 120                                                   \\ 
\hline
\multirow{3}{*}{CA-CondMat}             & \multirow{3}{*}{23133} & \multirow{3}{*}{93439}    & \multirow{3}{*}{3.5E-04} & \multicolumn{2}{c|}{\multirow{3}{*}{*}} & 2854.6          & 4454  & 1.5E+03         &      & 4908.2           & 4908.0  & 1.5E-02         & 5025.2          & \textbf{5025.3} & 1.1E+01         & 10                                                    \\
                                        &                        &                           &                          & \multicolumn{2}{c|}{}                   & 4747.9          & 4770  & 6.1E+01         & 1500 & 4908.4           & 4908.0  & 6.0E-11         & 5041.3          & \textbf{5041.3} & 4.5E-02         & 60                                                    \\
                                        &                        &                           &                          & \multicolumn{2}{c|}{}                   & 4828.2          & 4836  & 2.9E+01         &      & 4908.4           & 4908.0  & 5.8E-11         & 5041.4          & \textbf{5041.3} & 1.0E-01         & 120                                                   \\ 
\hline
\multirow{3}{*}{$G_{0.75}(10000)$}      & \multirow{3}{*}{10000} & \multirow{3}{*}{37491205} & \multirow{3}{*}{7.5E-01} & \multicolumn{2}{c|}{\multirow{3}{*}{*}} & 30049.8         & 30796 & 7.8E+02         &      & \textbf{34020.0} & 34008.7 & 6.7E+01         & 33995.6         & 33993.3         & 6.9E+01         & 10                                                    \\
                                        &                        &                           &                          & \multicolumn{2}{c|}{}                   & 30484.3         & 31640 & 9.1E+02         & 100  & \textbf{34191.8} & 34188.7 & 1.2E+01         & 34152.8         & 34152.0         & 1.2E+01         & 60                                                    \\
                                        &                        &                           &                          & \multicolumn{2}{c|}{}                   & 31015.7         & 32118 & 1.1E+03         &      & \textbf{34213.0} & 34212.7 & 3.8E+00         & 34174.8         & 34175.3         & 9.4E+00         & 120                                                   \\ 
\hline
\multirow{3}{*}{$G_{0.7}(8192)$}        & \multirow{3}{*}{8192}  & \multirow{3}{*}{23485837} & \multirow{3}{*}{7.0E-01} & \multicolumn{2}{c|}{\multirow{3}{*}{*}} & 28125.1         & 28882 & 8.8E+02         &      & \textbf{32371.7} & 32364.7 & 3.1E+01         & 32250.8         & 32248.7         & 4.8E+01         & 10                                                    \\
                                        &                        &                           &                          & \multicolumn{2}{c|}{}                   & 29105.0         & 30606 & 1.2E+03         & 100  & \textbf{32505.3} & 32506.0 & 7.2E+00         & 32404.5         & 32402.0         & 1.5E+01         & 60                                                    \\
                                        &                        &                           &                          & \multicolumn{2}{c|}{}                   & 30774.6         & 31392 & 8.4E+02         &      & \textbf{32534.8} & 32536.0 & 4.8E+00         & 32450.1         & 32451.3         & 6.7E+00         & 120                                                   \\ 
\hline
\multirow{3}{*}{$G_{0.8}(8192)$}        & \multirow{3}{*}{8192}  & \multirow{3}{*}{26840496} & \multirow{3}{*}{8.0E-01} & \multicolumn{2}{c|}{\multirow{3}{*}{*}} & 32104.3         & 32764 & 7.8E+02         &      & \textbf{35659.3} & 35659.3 & 3.4E+01         & 35575.1         & 35577.3         & 5.1E+01         & 10                                                    \\
                                        &                        &                           &                          & \multicolumn{2}{c|}{}                   & 32742.4         & 33556 & 9.3E+02         & 100  & \textbf{35754.3} & 35754.7 & 6.7E+00         & 35709.1         & 35708.7         & 9.7E+00         & 60                                                    \\
                                        &                        &                           &                          & \multicolumn{2}{c|}{}                   & 33582.0         & 34386 & 9.2E+02         &      & \textbf{35778.8} & 35778.0 & 4.9E+00         & 35723.1         & 35722.7         & 4.8E+00         & 120                                                   \\ 
\hline
\multirow{3}{*}{$P_{0.5}^{100}(4096)$}  & \multirow{3}{*}{4096}  & \multirow{3}{*}{4193301}  & \multirow{3}{*}{5.0E-01} & \multicolumn{2}{c|}{\multirow{3}{*}{*}} & 27091.0         & 27390 & 4.3E+02         &      & \textbf{27858.8} & 27860.0 & 6.6E+00         & 27804.3         & 27804.0         & 8.2E+00         & 10                                                    \\
                                        &                        &                           &                          & \multicolumn{2}{c|}{}                   & 27832.4         & 27832 & 1.7E+01         & 100  & \textbf{27901.2} & 27900.7 & 4.0E-01         & 27831.5         & 27831.3         & 2.5E+00         & 60                                                    \\
                                        &                        &                           &                          & \multicolumn{2}{c|}{}                   & 27858.5         & 27860 & 3.8E+00         &      & \textbf{27901.6} & 27901.3 & 5.6E-02         & 27844.2         & 27843.3         & 1.4E+00         & 120                                                   \\ 
\hline
\multirow{3}{*}{$P_{0.35}^{100}(8192)$} & \multirow{3}{*}{8192}  & \multirow{3}{*}{11747692} & \multirow{3}{*}{3.5E-01} & \multicolumn{2}{c|}{\multirow{3}{*}{*}} & 14452.4         & 17866 & 1.4E+03         &      & \textbf{23553.5} & 23563.3 & 4.4E+01         & 23338.4         & 23356.0         & 1.2E+02         & 10                                                    \\
                                        &                        &                           &                          & \multicolumn{2}{c|}{}                   & 23023.6         & 23334 & 4.2E+02         & 100  & \textbf{23639.4} & 23639.3 & 9.7E+00         & 23588.1         & 23589.3         & 2.1E+01         & 60                                                    \\
                                        &                        &                           &                          & \multicolumn{2}{c|}{}                   & 23591.6         & 23636 & 7.4E+01         &      & \textbf{23657.1} & 23656.7 & 7.0E+00         & 23627.6         & 23628.7         & 1.2E+01         & 120                                                   \\ 
\hline
\multirow{3}{*}{Wiki-Vote}              & \multirow{3}{*}{7115}  & \multirow{3}{*}{100762}   & \multirow{3}{*}{4.0E-03} & \multicolumn{2}{c|}{\multirow{2}{*}{*}} & 14194.0         & 14260 & 1.0E+02         &      & \textbf{14533.4} & 14534.0 & 6.2E+00         & 14389.9         & 14390.7         & 1.2E+01         & 10                                                    \\
                                        &                        &                           &                          & \multicolumn{2}{c|}{}                   & 14376.1         & 14374 & 2.2E+01         & 1500  & \textbf{14570.0} & 14570.0 & 1.6E-05         & 14562.7         & 14562.0         & 2.2E+00         & 60                                                    \\
                                        &                        &                           &                          & 148.1         & 2282                    & 14412.8         & 14414 & 1.6E+01         &      & \textbf{14570.0} & 14570.0 & 5.1E-03         & 14568.9         & 14568.7         & 4.2E-01         & 120                                                   \\ 
\hline
\multirow{3}{*}{brock800-1}             & \multirow{3}{*}{800}   & \multirow{3}{*}{207505}   & \multirow{3}{*}{6.5E-01} & 26002.5       & 27622                   & 29052.8         & 29052 & 3.8E-06         &      & \textbf{29054.2} & 29054.0 & 2.5E-10         & 29048.0         & 29047.3         & 5.1E-07         & 10                                                    \\
                                        &                        &                           &                          & 26098.1       & 27796                   & 29052.8         & 29052 & 3.5E-11         & 100  & \textbf{29054.2} & 29054.0 & 1.2E-09         & 29048.0         & 29047.3         & -8.6E-10        & 60                                                    \\
                                        &                        &                           &                          & 26109.0       & 27884                   & 29052.8         & 29052 & 1.9E-10         &      & \textbf{29054.2} & 29054.0 & 4.2E-09         & 29048.0         & 29047.3         & -5.0E-10        & 120                                                   \\ 
\hline
\multirow{3}{*}{p2p-Gnutella04}         & \multirow{3}{*}{10876} & \multirow{3}{*}{39994}    & \multirow{3}{*}{6.8E-04} & \multicolumn{2}{c|}{\multirow{3}{*}{*}} & 2082.0          & 2110  & 3.3E+01         &      & \textbf{2124.7}  & 2124.7  & 2.2E-11         & 2112.3          & 2113.3          & 1.8E+00         & 10                                                    \\
                                        &                        &                           &                          & \multicolumn{2}{c|}{}                   & 2124.0          & 2124  & 6.4E-01         & 1500 & \textbf{2124.7}  & 2124.7  & 2.2E-11         & 2118.0          & 2118.0          & -2.5E-12        & 60                                                    \\
                                        &                        &                           &                          & \multicolumn{2}{c|}{}                   & \textbf{2140.7} & 2142  & 1.3E+00         &      & 2124.7           & 2124.7  & 2.2E-11         & 2118.0          & 2118.0          & -2.5E-12        & 120                                                   \\ 
\hline
\multirow{3}{*}{p2p-Gnutella09}         & \multirow{3}{*}{8114}  & \multirow{3}{*}{26013}    & \multirow{3}{*}{7.9E-04} & \multicolumn{2}{c|}{\multirow{3}{*}{*}} & 3228.9          & 3240  & 1.1E+01         &      & \textbf{3237.3}  & 3237.3  & -5.2E-12        & \textbf{3237.3} & 3237.3          & -6.7E-12        & 10                                                    \\
                                        &                        &                           &                          & \multicolumn{2}{c|}{}                   & \textbf{3240.0} & 3240  & 4.5E-12         & 1500 & 3237.3           & 3237.3  & -3.3E-12        & 3237.3          & 3237.3          & -7.0E-12        & 60                                                    \\
                                        &                        &                           &                          & \multicolumn{2}{c|}{}                   & \textbf{3240.0} & 3240  & 4.5E-12         &      & 3237.3           & 3237.3  & -3.6E-12        & 3237.3          & 3237.3          & -7.3E-12        & 120                                                   \\
\hline
\end{tabular}
\end{table}
\end{landscape}
\restoregeometry

\subsection{The eigenvalue complementarity problem}
\label{subsection_Eicp}

We study the performance of the \qRCCD{} algorithm on the \textit{eigenvalue complementarity (EiC) problem}. 
\citeauthor{patrascu2015efficient}~\cite{patrascu2015efficient} also tested their 2-random coordinate descent  algorithm on the EiC problem.  The numerical results in~\cite{patrascu2015efficient} show that the 2-RCD algorithm provides better performance in terms of objective and computational times than the efficient algorithm for the  EiC problem from~\cite{lethi:hal-01636682}. {We also compare our algorithm with PGM and the deterministic 2-coordinate descent algorithm from \cite{beck20142}.}

The EiC problem is  an extension of the classical eigenvalue problem whose non-convex logarithmic formulation is as follows (see~\cite{RaydanRosaSantos}):
\begin{equation}
\label{eqn_EiCP}
\begin{aligned}
\max_{x \in \mathbb{R}^n} \quad & f(x) = \ln{ \left( \frac{x^\top A x}{x^\top B x } \right)}\\
\textrm{s.t.} \quad & \mathbf{1}^\top x = 1, \, x \geq 0,
\end{aligned}
\end{equation}
where we assume that both $x^\top A x$ and $x^\top B x$ are positive for all feasible $x$. 
For some $J \subseteq [n]$ and any matrix $A$, we define $A_J$ as the submatrix of $A$ generated from the rows and columns of $A$ indicated by $J$. Then, as proven in  \cite[Lemma 9]{patrascu2015efficient}, function $f$ from \eqref{eqn_EiCP}, {with $A$ and $B$ elements of \eqref{eqn_feasibleMatSet}},  satisfies Lipschitz condition
\eqref{q-coordinate} for
\begin{align}
    \label{eqn_LipschitzEicp}
    L_J = 2 \left( \frac{ \| A_J \| }{x^\top A x} + \frac{ \| B_J \| }{x^\top B x} \right) \leq {2n \left( \frac{\| A_J \| }{\min_{i \in [n]} A_{ii}} + \frac{\| B_J \| }{\min_{i \in [n]} B_{ii}}\right).}
\end{align}

{In the algorithm from \cite{patrascu2015efficient}, the above upper bound on $L_J$ is taken as Lipschitz constant. In contrast, in our algorithm we take $L_J$ as 
Lipschitz constant}, and take for $\| \cdot \|$ the 1-norm (that is, the largest absolute column sum).

 It is not difficult to verify that 
\begin{align}
    \label{eqn_NableEicp}
    \nabla f(x) = 2 \Big( \frac{1}{x^\top A x}A - \frac{1}{x^\top B x}B \Big) x.
\end{align}
Note that for the update \eqref{eqn_closedFormUpdate}, we only need $\nabla_J f$, and so it is not necessary to compute all the entries of $Ax$ and $Bx$ in \eqref{eqn_NableEicp}, but only those indicated by $J$. 

{ To compute \eqref{eqn_closedFormUpdate}, we use the algorithm proposed in \cite{condat2016fast}, for projecting onto the (scaled) simplex.  This algorithm has worst case complexity  $\mathcal{O}(q^2)$. However, its complexity is observed linear in practice, see~\cite{condat2016fast}.}

We test the \qRCCD{} algorithm on the EiC problem, for symmetric matrices $A$ and $B$ from the following set:
\begin{equation}
    \label{eqn_feasibleMatSet}
    \left\{ A \in \mathbb{R}^{n \times n} : \, \,
    \begin{aligned}
        & A_{ij} = A_{ji} \geq 0, \, \, A_{ii} = 0.001 + |Z_i| \,\, \forall i,j \in [n], \\[0.5ex]
        &  Z_i \text{ standard normal random variable } \forall i \in [n]
    \end{aligned}
    \right\}
\end{equation}
Specifically, we  generate $A$ and $B$ from the set \eqref{eqn_feasibleMatSet} such that their nonzero off-diagonal entries are randomly drawn uniform variables on the range $(0,1]$. This ensures that $x^\top A x > 0$ and $x^\top B x > 0$, for all feasible $x$. Hence, $f(x)$, $L_J$ and $\nabla f (x)$, see \eqref{eqn_EiCP}, \eqref{eqn_LipschitzEicp} and \eqref{eqn_NableEicp}, are well defined for all feasible $x$. We denote by $d$ the density of matrices  (i.e., the number of nonzeros divided by $n^2$), which we vary across the numerical experiments.

In \Cref{table_qRccaEic}, we take $n = 10^5$ and $d = 10^{-4}$, and report the objective values attained by the \qRCCD{} algorithm, averaged over 50 runs, for different values of $q$ and  the number of iterations. We report the determined lower bound of \eqref{eqn_EiCP} as $e^{f(x)}$ and the required computation time in seconds. Note that, while the number of iterations times $q$ is constant per row, the required computation time is not. {Based on \Cref{table_qRccaEic}, larger values of $q$, i.e., $q \geq 20$, yield both better bounds and lower computation times. This is a result of the low density $d$: since $d$ is low, computing gradients of size $q$, with $q$ large, is relatively cheap, and so the results improve when taking larger gradient vectors (even the PGM, with $q = n$, performs well on such sparse matrices). When $q$ is small, each iteration requires less time, but also brings little improvement to the objective value. Since every iteration of the \qRCCD{} algorithm incurs considerable overhead from indexing matrices $A$ and $B$ based on $J$, a large value of $q$ performs better. A large $q$ ensures that you perform less iterations in the same time as a small $q$, and thus lose less overhead time. Moreover, the lower number of iterations compared to small $q$, is compensated  by the fact that iterations for large $q$ improve the objective value more, compared to small $q$.
\begin{table}[]
\centering
\caption{The \qRCCD{} algorithm for the EiC problem averaged over 50 runs ($n=10^5, \, d = 10^{-4})$}
\label{table_qRccaEic}
\begin{tabular}{l|llllll|c}
\hline
$q$      & 2       & 5       & 20     & 50     & 100    & 200    & \multicolumn{1}{l}{\textbf{Iter.}} \\ \hline
Bound & 46.323 & 56.467  & 73.887  & 74.130  & 72.251  & 68.403  & \multirow{2}{*}{$1 \cdot 10^7 / q$} \\
Time (s) & 124.130 & 73.017  & 23.055 & 11.461 & 7.705  & 5.984  &                                    \\ \hline
Bound & 73.526 & 84.044  & 112.827 & 110.987 & 111.608 & 105.747 & \multirow{2}{*}{$2 \cdot 10^7 / q$} \\
Time (s) & 239.199 & 144.536 & 45.873 & 22.951 & 15.376 & 11.981 &                                    \\ \hline
Bound & 91.455 & 103.666 & 141.461 & 137.785 & 140.309 & 132.169 & \multirow{2}{*}{$3 \cdot 10^7 / q$} \\
Time (s) & 353.920 & 215.922 & 68.906 & 34.461 & 23.076 & 17.974 &                                    \\ \hline
\end{tabular}
\end{table}

Next, we investigate the convergence properties of the \qRCCD{} algorithm towards a stationary point. For this purpose, we generate two matrices $A$ and $B$ from \eqref{eqn_feasibleMatSet}, with $n = 3000$, and density $d = 10^{-3}$. We set $q = 750$ and run the \qRCCD{} algorithm for $7.5 \cdot 10^5$ iterations, while tracking the value of $\log_{10}(\widetilde{M})$, see \eqref{eqn_wideTildeMDef}, during these iterations\footnote{It can occur that $\widetilde{M} = 0$, in which case the logarithm is not defined. To circumvent this, we instead track the value of $\log_{10}(\max\{ \widetilde{M}, 10^{-16} \})$.}. We repeat this procedure $50$ times, and present the averaged results in \Cref{figure_ConvergenceM}. The figure clearly demonstrate convergence of the algorithm towards a stationary point.}

\begin{figure}[h]
    \centering
    \caption{Convergence of the \qRCCD{} algorithm $(n = 3000, \, d = 10^{-3}, \, q =750)$}
    \label{figure_ConvergenceM}
\begin{tikzpicture}[trim axis left, trim axis right]
\begin{axis}[
    xlabel={Number of iterations ($10^5$)},
    ylabel={$\log_{10}(\widetilde{M}) \quad \quad \quad$},
    ylabel style={rotate=-90},
    ymin=-7, ymax=1,
    xmin=0, xmax=7.5,
    ytick={-7,-6,-5,-4,-3,-2,2,-1,0,1},
    xtick={0,1,2,3,4,5,6,7},
    legend pos=north west,
    ymajorgrids=true,
    grid style=dashed,
]

\addplot[
    color=black,
    ]
    coordinates {
    (0.01,0.79)(0.01,0.58)(0.02,0.54)(0.03,0.55)(0.04,0.53)(0.05,0.52)(0.06,0.50)(0.07,0.48)(0.08,0.44)(0.09,0.40)(0.10,0.36)(0.11,0.31)(0.12,0.27)(0.13,0.22)(0.14,0.17)(0.15,0.12)(0.16,0.07)(0.17,0.02)(0.18,-0.03)(0.19,-0.08)(0.20,-0.13)(0.21,-0.17)(0.22,-0.21)(0.23,-0.26)(0.24,-0.29)(0.25,-0.33)(0.26,-0.37)(0.27,-0.40)(0.28,-0.44)(0.29,-0.47)(0.30,-0.50)(0.31,-0.53)(0.32,-0.56)(0.33,-0.58)(0.34,-0.61)(0.35,-0.64)(0.36,-0.66)(0.37,-0.69)(0.38,-0.71)(0.39,-0.74)(0.40,-0.76)(0.50,-0.96)(0.60,-1.14)(0.70,-1.29)(0.80,-1.42)(0.90,-1.54)(1.00,-1.65)(1.10,-1.76)(1.20,-1.86)(1.30,-1.95)(1.40,-2.04)(1.50,-2.12)(1.60,-2.20)(1.70,-2.28)(1.80,-2.36)(1.90,-2.43)(2.00,-2.50)(2.10,-2.57)(2.20,-2.64)(2.30,-2.71)(2.40,-2.77)(2.50,-2.84)(2.60,-2.90)(2.70,-2.96)(2.80,-3.02)(2.90,-3.08)(3.00,-3.14)(3.10,-3.20)(3.20,-3.26)(3.30,-3.32)(3.40,-3.38)(3.50,-3.43)(3.60,-3.49)(3.70,-3.55)(3.80,-3.61)(3.90,-3.66)(4.00,-3.72)(4.10,-3.77)(4.20,-3.83)(4.30,-3.89)(4.40,-3.94)(4.50,-4.00)(4.60,-4.05)(4.70,-4.11)(4.80,-4.17)(4.90,-4.22)(5.00,-4.28)(5.10,-4.34)(5.20,-4.39)(5.30,-4.45)(5.40,-4.51)(5.50,-4.57)(5.60,-4.63)(5.70,-4.69)(5.80,-4.75)(5.90,-4.81)(6.00,-4.87)(6.10,-4.93)(6.20,-4.99)(6.30,-5.06)(6.40,-5.12)(6.50,-5.19)(6.60,-5.26)(6.70,-5.32)(6.80,-5.39)(6.90,-5.47)(7.00,-5.54)(7.10,-5.61)(7.20,-5.69)(7.30,-5.77)(7.40,-5.85)(7.50,-5.93)
    };
\end{axis}
\end{tikzpicture}
\end{figure}
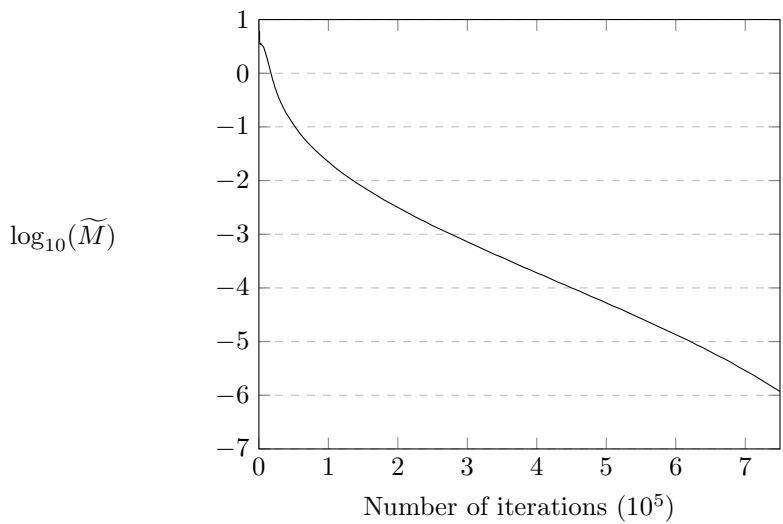

{Lastly, in \Cref{big_eic_table}, we investigate the performance of the \qRCCD{} algorithm compared to PGM, 2-RCD and the deterministic 2-coordinate descent algorithm by Amir \citeauthor{beck20142} \cite{beck20142}, abbreviated as \AbAlg{}. At each iteration, the \AbAlg{} algorithm picks the two coordinates that, when updated, decrease a specific stationarity measure the most. Let us briefly present the approach used in \AbAlg{} for the index selection, specifically for the maximization problem \eqref{eqn_EiCP}. \AbAlg{} computes first $\argmax_i \, (\nabla f(x))_i$, for  the current solution $x$. Then \AbAlg{} constructs another vector that depends on $\max_i \, (\nabla f(x))_i$, and  computes the index corresponding to the largest element of this new vector. In the EiC problem, both these vectors are dense and $n$-dimensional.

We test these algorithms on six different sets of matrices $A$ and $B$ from \eqref{eqn_feasibleMatSet}, with varying sizes $n$ and densities $d$, which refers to the density of both $A$ and $B$. The number of nonzeros in matrix is denoted $\texttt{nnz}$, and is given by $n^2 d$. Note that the tested matrices are much more dense than those from numerical results in \cite{patrascu2015efficient}. We report the attained objective values as $e^{f(x)}$, see \eqref{eqn_EiCP}, per method at 15, 30 and 60 minutes. For the \qRCCD{} and 2-RCD algorithms, we average the objective values over five runs, and have set $q = 50$ and blocksize ${\mathcal{B}} = 25$, which seemed to attain high objective values based on preliminary numerical tests. We display the results in \Cref{big_eic_table}, and mark the best objective value in boldface.

For $n = 10^5$, \AbAlg{} performs well: it converges to a stationary point in 15 minutes, but it is unable to escape this local maximum. In the same instances, the convergence of \qRCCD{} to stationary points is slower. However, due to its random element, the  \qRCCD{} algorithm 
may find more favorable (with higher objective) stationary points.
This allows the \qRCCD{} algorithm to outperform the \AbAlg{} in the first instance at 60 minutes.

For all three instances with $n = 10^5$, the \qRCCD{} algorithm is unable to converge to any stationary point within the time limit. This claim is justified as follows: the output of our algorithm shows that 
the objective values of the \qRCCD{} algorithm at minute 60, was noticeably higher compared to the objective values at minute 59 (average increase of $ 1.01 \% $ over all $n = 10^5$ instances).

When $n = 10^6$ or $n = 10^7$, \qRCCD{} is clearly the superior algorithm. 
The iterations of the \AbAlg{} and PGM algorithms become too costly
due to their operations on dense $n$-dimensional vectors. The computational cost of 2-RCD is similar to that of \qRCCD{}. However, it is still consistently outperformed by the \qRCCD{} algorithm, as the data shows.
}

\begin{table}[]
\centering
\caption{Objective values of different methods for the large scale EiC problem ($q = 50, \, {\mathcal{B}} = 25$)}
\label{big_eic_table}
\begin{tabular}{c|c|c|l|l|l|l|l}
\hline
\multicolumn{1}{l|}{$\log_{10}(n)$} & \multicolumn{1}{l|}{$d$} & \multicolumn{1}{l|}{$\texttt{nnz} (:= n^2 d)$} & \AbAlg{} & PGM & \qRCCD{} & 2-RCD & \textbf{T. (min)} \\ \hline
\multirow{3}{*}{5} & \multirow{3}{*}{$\frac{1}{2} \cdot 10^{-2}$} & \multirow{3}{*}{$\frac{1}{2} \cdot 10^{8}$} & \textbf{852.80} & 79.57 & 494.29 & 286.62 & 15 \\
 &  &  & \textbf{852.80} & 120.43 & 734.96 & 426.91 & 30 \\
 &  &  & 852.80 & 178.42 & \textbf{952.72} & 588.05 & 60 \\ \hline
\multirow{3}{*}{5} & \multirow{3}{*}{$10^{-2}$} & \multirow{3}{*}{$10^{8}$} & \textbf{2202.94} & 34.00 & 240.89 & 195.39 & 15 \\
 &  &  & \textbf{2202.94} & 53.30 & 383.76 & 296.83 & 30 \\
 &  &  & \textbf{2202.94} & 80.81 & 590.80 & 405.25 & 60 \\ \hline
\multirow{3}{*}{5} & \multirow{3}{*}{$2 \cdot 10^{-2}$} & \multirow{3}{*}{$2 \cdot 10^{8}$} & \textbf{1295.42} & 9.99 & 147.98 & 88.59 & 15 \\
 &  &  & \textbf{1295.42} & 19.60 & 253.13 & 142.84 & 30 \\
 &  &  & \textbf{1295.42} & 30.95 & 387.43 & 241.92 & 60 \\ \hline
\multirow{3}{*}{6} & \multirow{3}{*}{$10^{-4}$} & \multirow{3}{*}{$10^{8}$} & 1.39 & 14.44 & \textbf{49.53} & 44.52 & 15 \\
 &  &  & 2.56 & 26.24 & \textbf{80.49} & 70.41 & 30 \\
 &  &  & 5.60 & 43.11 & \textbf{136.95} & 114.85 & 60 \\ \hline
\multirow{3}{*}{6} & \multirow{3}{*}{$10^{-3}$} & \multirow{3}{*}{$10^{9}$} & 1.03 & 1.06 & \textbf{5.86} & 4.87 & 15 \\
 &  &  & 1.15 & 1.13 & \textbf{12.19} & 8.78 & 30 \\
 &  &  & 1.56 & 1.30 & \textbf{23.16} & 15.99 & 60 \\ \hline
\multirow{3}{*}{7} & \multirow{3}{*}{$10^{-5}$} & \multirow{3}{*}{$10^{9}$} & 1.00 & 1.40 & \textbf{7.45} & 6.94 & 15 \\
 &  &  & 1.01 & 1.98 & \textbf{15.26} & 14.38 & 30 \\
 &  &  & 1.04 & 3.73 & \textbf{25.75} & 23.78 & 60 \\ \hline
\end{tabular}
\end{table}

\section{Conclusions}
In this paper, we propose a random coordinate descent algorithm for minimizing a non-convex objective function subject to one linear constraint and bounds on the variables, see~\Cref{Alg1}.
The \qRCCD{} algorithm  randomly selects $q$ ($q \geq 2$) variables and updates them  based on the solution of an appropriate convex optimization problem, see \eqref{update_subproblem}.
We prove that any accumulation point of 
the sequence generated by the \qRCCD{} algorithm, using the uniform distribution, is a stationary point for~\eqref{general}, see~\Cref{thm:accumulation point}.  The convergence rate for the expected values of the {stationarity measure} is given in~\Cref{thm:convergenceRate}.
The convergence rates of  
the $2$-RCCD algorithm  and the algorithm from~\cite{patrascu2015efficient}, for the case that blocks are of size one, coincide. The  convergence rate of  the~\qRCCD{} algorithm improves for $q>2$.

We test the \qRCCD{} algorithm for solving the \BlueUpdate{D$k$S} problem and the \BlueUpdate{EiC} problem in \Cref{sect:numerics}. Extensive numerical results show that the quality of bounds improves as $q$ increases,  and that the~\qRCCD{} algorithm converges towards a stationary point. \BlueUpdate{In both these problems, the \qRCCD{} algorithm compares favourably with the other algorithms. In particular, \Cref{big_eic_table} demonstrates that for large scale EiC problems ($n \geq 10^6$), the \qRCCD{} algorithm significantly outperforms \AbAlg{} and PGM, and also improves  on the 2-RCD algorithm.}

\BlueUpdate{\medskip\medskip

\noindent \textbf{Acknowledgements.} We thank two anonymous reviewers for providing us with insightful feedback on earlier versions of this manuscript.\\

\noindent \textbf{Conflict of interest statement.} The authors have no conflicts of interest to declare.}

\bibliography{refs}
\end{document}